\documentclass[12pt]{amsart}
\usepackage{graphics}
\usepackage{amsmath}
\usepackage{amsfonts,amscd}
\usepackage{amsthm}
\usepackage{amssymb}
\usepackage{latexsym}
\usepackage{wasysym}
\usepackage{mathrsfs}
\usepackage{yfonts}
\usepackage{stmaryrd}
\usepackage{url}
\usepackage{bm}
\usepackage{resizegather}
\usepackage{verbatim}
\usepackage{graphicx}
\usepackage{epsf}
\usepackage{enumerate}
\usepackage{geometry}
\usepackage{color}
\usepackage{hyperref}
\definecolor{dark-blue}{rgb}{0,0,0.6}
\definecolor{Purple}{rgb}{0.2,0,0.25}
\hypersetup{breaklinks=true,
pagecolor=white,
colorlinks=true,%false
citecolor=dark-blue,%{rgb}{0.2,0.2,1},%
filecolor=black,%
linkcolor=Purple,%
urlcolor=blue
}

\newtheorem{thm}{Theorem}
\newtheorem{condition}[thm]{Condition}
\newtheorem{cor}[thm]{Corollary}
\newtheorem{lem}[thm]{Lemma}
\newtheorem{prop}[thm]{Proposition}
\newtheorem{algorithm}[thm]{Algorithm}

\newtheorem{defin}[thm]{Definition}

\theoremstyle{definition}
\newtheorem{expl}[thm]{Example}
\newtheorem{rem}[thm]{Remark}

\numberwithin{equation}{section}
\numberwithin{thm}{section}

\newcommand{\wt}{\widetilde}

\newcommand{\R}{\mathbb{R}}

\newcommand{\N}{\mathbb{N}}

\newcommand{\X}{\mathbf{X}}

\newcommand{\wh}{\widehat}

\newcommand{\bref}[1]{\textbf{\ref{#1}}} %bold font for any cross reference 
\newcommand{\beqref}[1]{\textbf{(\ref{#1})}} %bold font for any equation number
\begin{document}

\title[New results related to cutters]{New results related to cutters and to an extrapolated block-iterative method for finding a common fixed point of a collection of them}

\author[Yair Censor]{Yair Censor}
\address{Department of Mathematics, University of Haifa, Mt. Carmel, 3498838 Haifa, Israel. } 
\email{yair@math.haifa.ac.il}
\author[Daniel Reem]{Daniel Reem}
\address{The Center for Mathematics and Scientific Computation (CMSC), University of Haifa, Mt. Carmel, 3498838 Haifa,   Israel. }
\email{dream@math.haifa.ac.il} 
\author[Maroun Zaknoon]{Maroun Zaknoon}
\address{Department of Mathematics, The Arab Academic College for Education, 22 HaHashmal Street,  32623 Haifa, Israel.}
\email{zaknoon@arabcol.ac.il}

\setcounter{page}{1}% to be added afterwards
%\coordinates{vv}{yyyy}{ii}{ppp-ppp}% to be added afterwards
\date{September 4, 2025}% to be added afterwards % submission date 
%\undate{April 21, 2025}% to be added afterwards % revised date 

%\corauthor{Daniel Reem, dream@math.haifa.ac.il}

\begin{abstract}
Given a Hilbert space and a finite family of operators defined on the space, the common fixed point problem (CFPP) is to find a point in the intersection of the fixed point sets of these operators.  Instances of the problem have numerous applications in science and engineering. We consider an extrapolated block-iterative method with dynamic weights for solving the CFPP assuming the operators belong to a wide class of operators called cutters. Global convergence is proved in two different scenarios, one of them is under a seemingly new condition on the weights which is less restrictive than a condition suggested in previous works. In order to establish convergence, we derive various new results of independent interest related to cutters, some of them extend, generalize and clarify previously published results.
\end{abstract}

\maketitle
\pagestyle{myheadings}
\markboth{Censor, Reem, Zaknoon}{New results related to cutters}

\section{Introduction}\label{sec:Intro}
\subsection{Background: }

Given a space $X$ and a finite family of operators $T_{i}: X\to X$, $i\in I:=\{1,2,\ldots,m\}$, $m\in\N$, the common fixed point problem (CFPP) is the problem of finding a point in the intersection of the fixed point sets of these operators, assuming that the intersection is nonempty. In other words, the CFPP is the following problem:  
\begin{equation}\label{eq:CFPP}
\textnormal{Find}\,\,x\in \bigcap_{i\in I} F_i, 
\end{equation}
where $F_i:=\textnormal{Fix}(T_i):=\{x\in X\,|\,T_i(x)=x\}$ for every $i\in I$. Various methods have been devised to solve the CFPP, under certain assumptions, such as the ones described in  \cite{BauschkeCombettes2017book,Cegielski2012book,
CensorReemZaknoon2022jour,CensorSegal2009jour(CFPP),
ReichZalas2016jour,Zaslavski2016book} and in some of the references therein; see also the last paragraph of this subsection (for the case $m=1$, namely when one wants to find, or approximate, a fixed point of a given mapping $T_1:X\to X$ in various settings, or to consider related aspects such as existence or  uniqueness, see, for instance, \cite{Berinde2007book,
GranasDugondji2003book,HandbookMetricFixedPoint2001book,
RusPetruselPetrusel2008book,YounisChenSingh2024book} and the references therein). A particular instance of this problem, when $X$ is a Hilbert space and each operator $T_i$ is the orthogonal projection onto a closed and convex subset $C_i$,  $i\in I$, is the so-called convex feasibility problem (CFP), which has found theoretical and real-world applications in image reconstruction, radiation therapy treatment planning, signal processing, and more: see, for instance, \cite[p. 23]{Cegielski2012book} for a long and not exhaustive list of such applications and  \cite{AharoniCensor1989jour,BauschkeBorwein1996jour,
Bregman1967jour,Cegielski2012book,CensorReem2015jour,
CensorZenios1997book,Combettes1996jour(CFP)} for a few methods for solving the CFP. 

In \cite{CensorReemZaknoon2022jour} we considered the CFPP under the assumptions that $X$ was a Euclidean space (namely a finite-dimensional real Hilbert space) and each operator $T_{i}$, $i\in I$ was a  continuous cutter. The class of cutters was  introduced (with a different terminology: see the lines after \beqref{eq:FixCutter} below) by Bauschke and Combettes in \cite{BauschkeCombettes2001jour} and by Combettes in \cite{Combettes2001inproc}. This is a  rather wide class which includes, in particular,  orthogonal projections, firmly nonexpansive operators, subgradient projections of 
continuous convex functions having nonempty zero-level-sets, and resolvents of maximally monotone operators \cite[Proposition 2.3]{BauschkeCombettes2001jour}. See also  \cite{BauschkeBorweinCombettes2003jour,
BauschkeCombettes2017book,BauschkeWangWangXu2015jour,
Cegielski2012book,
CegielskiCensor2011inproc,CensorReemZaknoon2022jour,
CensorSegal2009jour(split),KimuraSaejung2015jour,
MillanLindstromRoshchina2020inproc,WangXu2011jour,
Zaknoon2003PhD} and Sections \bref{sec:Preliminaries}--\bref{sec:Auxiliary} below for more details, examples, results, variants and generalizations related to cutters. The algorithmic scheme that we used in \cite[Algorithm 1]{CensorReemZaknoon2022jour} was inspired by the BIP (Block-Iterative Projections) method of Aharoni and Censor \cite{AharoniCensor1989jour}. In \cite{CensorReemZaknoon2022jour} we  generalized the BIP method from orthogonal projections to continuous cutters, allowed dynamic weights and certain adaptive perturbations, and still concluded that the whole algorithmic sequence converges globally to a common fixed point of the given cutters (see \cite[Theorem 5.8.15]{Cegielski2012book}, \cite[Theorem 9.27]{CegielskiCensor2011inproc}, \cite[Theorems 4.1 and 4.5] {KolobovReichZalas2017jour}, \cite[Theorem 3.2]{NimanaArtsawang2024jour}, and \cite[Theorems 3.1 and 3.2]{ReichZalas2016jour} for related, but somewhat different, results). 

In this work we continue the study done in \cite{CensorReemZaknoon2022jour} by weakening one of the conditions made there about the relaxation parameters, namely we allow them to belong to a wider interval than the typically used interval $[\tau_1,2-\tau_2]$, where $\tau_1,\tau_2\in (0,1]$. In other words, we allow extrapolation. The appearance of extrapolation, while intended to accelerate the convergence of the algorithmic sequence by allowing deeper steps, introduces complications which do not appear without extrapolation, and it forces us to make a more restrictive assumption than in  \cite{CensorReemZaknoon2022jour} on either the dynamic weights or on the common fixed point set in order to ensure global convergence. More precisely, we either require the weights to satisfy a certain intermittent condition (Condition \bref{cond:sw} below) instead of imposing  the condition $\sum_{k=0}^{\infty}w_k(i)=\infty$ for each $i\in I$ (Condition \bref{cond:sum_infty} below) as in \cite{CensorReemZaknoon2022jour}  (following \cite{AharoniCensor1989jour}), or we require the common fixed point set to have a nonempty interior (and then we allow Condition \bref{cond:sum_infty}). 

We note that several works, apparently starting from \cite{Merzljakov1962jour}, study extrapolation in the context of the CFPP in various settings (possibly in the particular case of the CFP, possibly in the linear case): see, for instance, \cite{AleynerReich2008jour,BargetzKolobovReichZalas2019jour,
BauschkeCombettesKruk2006jour,BuongAnh2023jour,BuongNguyen2023jour,
Cegielski2012book,CegielskiCensor2012jour,CegielskiNimana2019jour,
Combettes1996jour(CFP),Combettes1997jour-b,Combettes2001inproc,
CombettesWoodstock2021jour,Crombez2002jour,
Dos-Santos1987jour,Kiwiel1995jour,Merzljakov1962jour,
NikazadMirzapour2017jour,Ottavy1988jour,Pierra1984jour} for a partial list of such works. However, as far as we know, only  \cite{BargetzKolobovReichZalas2019jour,BauschkeBorweinCombettes2003jour,BauschkeCombettesKruk2006jour,
BuongAnh2023jour,Combettes2001inproc} study extrapolated algorithms in the spirit of our algorithm in the context of the CFPP with dynamic weights and general cutters (see also \cite[Section 3]{BuongNguyen2023jour}, \cite[Section 6.5]{Combettes2000jour}, and \cite[Section 3]{CombettesWoodstock2021jour} for closely related by somewhat different extrapolation algorithms with dynamic weights and cutters), and there the condition on the weights is either less general than Condition \bref{cond:sw} below or is not implied by it and does not imply it, as we explain in Remark \bref{rem:ComparisonConditions} and Examples \bref{ex:1/(2m)}--\bref{ex:1/m} below. Hence, our convergence theorems are not implied by the convergence results mentioned in \cite{BargetzKolobovReichZalas2019jour,BauschkeBorweinCombettes2003jour,BauschkeCombettesKruk2006jour,
BuongAnh2023jour,Combettes2001inproc} (and elsewhere). It should be noted, however, that with the exception of the assumptions on the weights, the settings in \cite{BargetzKolobovReichZalas2019jour,BauschkeBorweinCombettes2003jour,BauschkeCombettesKruk2006jour,
BuongAnh2023jour,Combettes2001inproc} are  more general than the setting that we consider: for instance, the space there might be infinite-dimensional, the cutters need not be continuous  (but rather weakly regular, namely demi-closed at 0), in some of these works appear either strings of cutters (while we consider strings of length 1) or perturbations, and the algorithmic schemes have more general forms than the form of Algorithm \bref{alg:Extrapolation} below. 

We wish to end this subsection by commenting about potential  computational advantages of our work. Our paper is a theoretical study which does not attempt to show that the method that we discuss  is better than other methods (extrapolated or not) from the computational point of view, although it might be the case. In order to prove or disprove this latter possibility, or at least to understand it in a much better manner, one should make exhaustive tests of the many possible specific variants  and user-chosen parameters  permitted by the general scheme that we discuss, such as the relaxation parameters, the operators and the weights, and, preferably, to include significant real-world applications rather than simple and not-very-convincing demonstrative examples that do not allow to methodologically explore all possibilities. The additional option for extrapolation adds yet another layer of mathematical generality with possible numerical ramifications. All in all, the above-mentioned computational task is beyond our current abilities and beyond the scope of this paper, but we express our hope that the theoretical results that we presented will be tested numerically in the future in one way or another.

\subsection{Contribution:} The contribution of our work is two fold.  First, we obtain new results related to cutters such as Corollary \bref{cor:T_i(x)=x}, Lemma \bref{lem:T(w,wh(lambda))} and Propositions \bref{prop:T_w neq x P_D(T(x)) neq x} and \bref{prop:FejerT_{w,lambda}} below, some of them extend and generalize previously published results, and we fill a certain gap which appears in \cite{Pierra1984jour} (see Remark \bref{rem:P_D(P_C(x)) neq x} below); these results hold in general real Hilbert spaces and without the continuity assumption on the cutters.  Second, we use these results in order to obtain finite-dimensional convergence results (Theorems \bref{thm:NonemotyInterior} and \bref{thm:AlmostCyclic} below) related to the extrapolated block iterative method (Algorithm \bref{alg:Extrapolation} below) under conditions on the weights which have not been considered yet in the context of extrapolated algorithms for solving the CFPP induced by cutters. 

\subsection{Paper layout: }
In Section \bref{sec:Preliminaries} we introduce some preliminary details. In Section \bref{sec:Auxiliary}  we present several results related to cutters. The extrapolated algorithmic scheme is introduced in Section \bref{sec:AlgorithmConvergence}, where we also present conditions on the weights, compare these conditions with previously published conditions, and present the convergence theorems. The paper is concluded in Section \bref{sec:Conclusion}.

\section{Preliminaries}\label{sec:Preliminaries}

We use in the sequel the following notations and definitions. Unless otherwise stated, the ambient space is a real Hilbert space $X$ with an inner product $\langle\cdot,\cdot\rangle$ and an induced norm $\|\cdot\|$. Given $\rho>0$ and $x\in X$, we denote the closed ball with radius $\rho$ centered at $x$ by $B[x,\rho]:=\left\{ y\in 
X\mid \left\Vert x-y\right\Vert \leq \rho \right\}$. Given  a natural number $m$, define the index set $I$ as $I:=\left\{ 1,2,\ldots ,m\right\}$. A function $w:I\rightarrow \left[ 0,1\right] $ which satisfies $\sum_{i\in I}w\left( i\right) =1$ is called a \textit{weight function}, and if, in addition,  $w(i)>0$ for all $i\in I$, then $w$ is called a \textit{positive weight function}. We refer to the vector $(w(i))_{i\in I}\in \R^m$ as a \textit{weight vector}, and when $w$ is a positive weight function, then we refer to $(w(i))_{i\in I}$ as a \textit{positive weight vector}. Denote $\wh{I}_w:=\{i\in I\mid w\left( i\right) >0\}$ and let $\wh{w}:\wh{I}_w\to (0,1]$ be the restriction of $w$ to $\wh{I}_w$, that is, $\wh{w}(i):=w(i)>0$ for all $i\in \wh{I}_w$. Note that $\wh{I}_w\neq \emptyset$ (since $\sum_{i\in I}w(i)=1$), that $\wh{w}$ is a positive weight function defined on $\wh{I}_w$, and the case $\wh{w}(i)=1$ for some $i\in \wh{I}$ is possible if and only if $\wh{I}$ is a singleton. 

For a given weight function $w$ and given operators $T_i:X\to X$, $i\in I$, we denote for all $x\in X$
\begin{equation}
T_{w}\left( x\right) :=\sum_{i\in I}w\left( i\right) T_{i}\left( x\right) ,
\label{eq:T_w}
\end{equation}
and define 
\begin{equation}\label{eq:L(x,w)}
L(x,w):=\left\{
\begin{array}{ll}
\displaystyle{\frac{\sum_{i\in I}w(i)\|T_{i}(x) -x\|^{2}}{\|T_{w}(x)-x\|^{2}}}, & \text{if } x\neq T_{w}(x),\\  
1, &\text{ otherwise.}
\end{array}
\right.
\end{equation}
Note that $L(x,w)\geq 1$ because of its definition and the convexity of the square of the norm. For each $\lambda\in\R$ (sometimes referred to as the \textit{relaxation/acceleration parameter} whenever it is positive) and each $x\in X$, let

\begin{equation}
T_{w,\lambda }\left( x\right) :=x+\lambda \left( T_{w}\left( x\right)-x\right)=x+\lambda\sum_{i\in I}w(i)(T_i(x)-x).  \label{P (2.1.5)}
\end{equation}
We observe that $T_w=T_{w,1}$ and that $T_{w,\lambda}$ can be written in a block form using $\wh{I}_w$ and $\wh{w}$, namely $T_{w,\lambda}(x)=x+\lambda\sum_{i\in \wh{I}_w}\wh{w}(i)(T_i(x)-x)=T_{\wh{w},\lambda}(x)$ for all $x\in X$. In particular, $T_w=T_{\wh{w}}$. 
 
Given the index set $I$, suppose that $w$ is a positive weight function. Adopting Pierra's product space formalism \cite{Pierra1984jour}, we define the product space $\X:=X^m$. Its elements are $\mathbf{x}:=\left( x^{1},x^{2},\ldots ,x^{m}\right)$, where $x^{i}\in X$ for all $i\in I,$ and the inner product in $\X$ is defined for all $\mathbf{x},\mathbf{y}\in \X$ as 
\begin{equation}
\left\langle \left\langle \mathbf{x},\text{ }\mathbf{y}\right\rangle
\right\rangle_w :=\sum_{i=1}^{m}w\left( i\right) \left\langle x^{i},\text{ }%
y^{i}\right\rangle .  \label{eq:inner}
\end{equation}%
The norm, and the distance function in the product space are denoted by $|||\cdot |||_w$ and $d_w\left( \left( \cdot ,\cdot \right) \right) ,$
respectively. Observe our notational rule to use bold-face upright letters for quantities in the product space. The \textit{diagonal set} $\mathbf{D}$ in $\X$ is 
\begin{equation}
\mathbf{D}:=\left\{ \mathbf{x\in } \X\mid \text{ 
}x\in X,\text{\ }\mathbf{x}=\left( x,x,\ldots ,x\right)
\right\} .  \label{P (1.2.1)}
\end{equation}%
It is easy to see that $\mathbf{D}$ is a closed linear subspace of  $\X$. The canonical mapping $\mathbf{J}:X\rightarrow 
\mathbf{D,}$ is defined by $\mathbf{J}\left( x\right) :=\left( x,x,\ldots ,x\right) $ for every $x\in X$. 

Let $T_{1},T_{2},\ldots ,T_{m}$ be operators from $X$ to itself and let $F_{1},F_{2},\ldots ,F_{m}$ be their fixed point sets, respectively, that is, $F_i:=\textnormal{Fix}(F_i):=\{x\in X\,|\,T_i(x)=x\}$. We define $\mathbf{T}:\X\rightarrow \X$ by 
\begin{equation}
\mathbf{T}\left( \mathbf{x}\right) :=\left( T_{1}\left( x^{1}\right)
,T_{2}\left( x^{2}\right) ,\ldots ,T_{m}\left( x^{m}\right) \right) .
\label{eq:T}
\end{equation}
It is immediate to verify that the fixed point set $\mathbf{F}:=\textnormal{Fix}(\mathbf{T})$ of the operator $\mathbf{T}$ satisfies the following relation:
\begin{equation}
\mathbf{F}=F_{1}\times F_{2}\times \cdots
\times F_{m}. \label{eq:fixTF}
\end{equation}%
In addition, the following equivalence holds for each $x\in X$: 
\begin{equation}
\mathbf{J}\left( x\right) \in \mathbf{F}\cap \mathbf{D}\Leftrightarrow x\in F:=\bigcap _{i\in I}F_{i}.  \label{P (1.2.5)}
\end{equation}%
This means that \beqref{eq:CFPP} is equivalent to the following  problem: 
\begin{equation}
\text{Find a point }\mathbf{x}\text{ in }\mathbf{F}\cap \mathbf{D}\subseteq \X\mathbf{.}  \label{P problem in product space}
\end{equation}

For any pair $x,y\in X$, define the set
\begin{equation}
H\left( x,y\right) :=\left\{ u\in X\mid \left\langle x-y, u-y\right\rangle \leq 0\right\}   \label{eq:halfspace}
\end{equation}
which is a half-space unless $x=y$ (and then it is the whole space). 
Let $T:X\rightarrow X$ be an operator. The operator $T$ is called a \textit{cutter} if it satisfies the condition 
\begin{equation}
\emptyset\neq \textnormal{Fix}(T)\subseteq H\left( x,T\left( x\right) \right) ,\text{ for all } x\in X,  \label{BC-condition}
\end{equation}
or, equivalently, $\textnormal{Fix}(T)\neq \emptyset$ and
\begin{equation}
\text{for all }q\in \textnormal{Fix}(T)\, \text{ and all }x\in X\, \text{ one has }\left\langle x-T(x),q-T(x)\right\rangle \leq 0.  \label{eq:directop}
\end{equation}
The set of all fixed points of a cutter is closed and convex because 
\begin{equation}\label{eq:FixCutter}
\textnormal{Fix}(T)=\bigcap_{x\in X}H\left( x,T\left( x\right) \right),
\end{equation}
as shown in \cite[Proposition 2.6(ii)]{BauschkeCombettes2001jour}. As said in Section \bref{sec:Intro}, the class of cutters was introduced in \cite{BauschkeCombettes2001jour,Combettes2001inproc}, with a different terminology (``class $\textfrak{T}$''; the name ``cutter'' was suggested in  \cite{CegielskiCensor2011inproc}; other names are used in the literature for these operators, such as ``firmly quasinonexpansive'' \cite[Definition 4.1(iv) and Proposition 4.2(iv), pp. 69--70]{BauschkeCombettes2017book}, \cite[Definition 3.5]{CombettesWoodstock2021jour} and ``directed operators'' \cite{CensorSegal2009jour(split)}, \cite{Zaknoon2003PhD}), and there various properties and examples of cutters can be found. If $T_i:X\to X$, $i\in I$ are cutters, then their fixed point sets are nonempty and hence  $\mathbf{F}\neq \emptyset$ according to \beqref{eq:fixTF}. This fact, as well as the inequality
\begin{equation}
\left\langle \left\langle \mathbf{z}-\mathbf{T}(\mathbf{z}), \mathbf{q-T}\left( \mathbf{z}\right)\right\rangle \right\rangle_w
=\sum_{i=1}^{m}w\left( i\right) \left\langle z^{i}-T_{i}(z^{i}),q^{i}-T_{i}(z^{i})\right\rangle \leq 0,
\end{equation}
for all $\mathbf{z}\in \X$ and $\mathbf{q\in F,}$ imply that the operator $\mathbf{T,}$ defined by \beqref{eq:T}, is a cutter in the product space $\X$. 

\section{Results related to cutters}\label{sec:Auxiliary}

In this section we present various results related to cutters. We will use these results in Section \bref{sec:AlgorithmConvergence}. We continue with the notation introduced in Section \bref{sec:Preliminaries}, and recall that $X$ is a real Hilbert space.

The following lemma is essentially due to Pierra \cite[Lemma 1.1]{Pierra1984jour}. For the sake of completeness, and since both the setting and the formulation of  \cite[Lemma 1.1]{Pierra1984jour} are somewhat different from our ones, we present the proof of the lemma. 

\begin{lem}\label{lem:P_Q P_D} %\label{P c 1.2.1} 
Suppose that $\{Q_{i}\}_{i=1}^{m}$ are nonempty, closed and convex subsets of $X$, and let $\mathbf{Q}:=Q_{1}\times Q_{2}\times \cdots \times Q_{m}$. Then the orthogonal projections onto $\mathbf{Q}$ and onto $\mathbf{D}$, respectively, in the product space $\X$, satisfy, for any $\mathbf{x}:=\left(x^{1},x^{2},\ldots,x^{m}\right)\in \X$, the following two relations:  
\begin{equation}
\mathbf{P}_{\mathbf{Q}}\left( \mathbf{x}\right) =\left( P_{Q_{1}}\left(
x^{1}\right) ,P_{Q_{2}}\left( x^{2}\right) ,\ldots ,P_{Q_{m}}\left(
x^{m}\right) \right) ,  \label{P (1.2.7)}
\end{equation}
\begin{equation} \label{P (1.2.8)}
\mathbf{P}_{\mathbf{D}}\left( \mathbf{x}\right) =\mathbf{J}\left(
\sum_{i=1}^{m}w\left( i\right) x^{i}\right) . 
\end{equation}
\end{lem}

\begin{proof}
We start by proving \beqref{P (1.2.7)}. Since $f_1(\min\{f_2(\mathbf{z})\,|\,\mathbf{z}\in \mathbf{Q}\})=\min\{f_1(f_2(\mathbf{z}))\,|\,\mathbf{z}\in \mathbf{Q}\}$ for  $f_1(t):=t^2$ and $f_2(\mathbf{z}):=\|\mathbf{x-z}\|$ for all $t\in [0,\infty)$ and $\mathbf{z}\in \X$, and since $\min\{\sum_{i\in I}a_i\,|\,a_i\in A_i, i\in I\}=\sum_{i\in I}\min A_i$ for all subsets $A_i$, $i\in I$ of real numbers such that each one of  them has a minimum, we have, using the definition of the orthogonal projection, that  
\begin{align}
&|||\mathbf{x}-\mathbf{P}_{\mathbf{Q}}(\mathbf{x})|||_w^2
=\left(\min\{|||\mathbf{x}-\mathbf{z}|||_w\,|\, \mathbf{z}\in \mathbf{Q}\}\right)^2
=\min\{|||\mathbf{x}-\mathbf{z}|||_w^2\,|\, \mathbf{z}\in \mathbf{Q}\}\\\notag
&=\min\left\{\sum_{i\in I}w(i)\|x^i-z^i\|^2\,|\, z^i\in Q_i\,\,\forall i\in I\right\}
=\sum_{i\in I}\min\{w(i)\|x^i-z^i\|^2\,|\, z^i\in Q_i\}\\\notag
&=\sum_{i\in I}w(i)\min\{\|x^i-z^i\|^2\,|\, z^i\in Q_i\}
=\sum_{i\in I}w(i)\|x^i-P_{Q_i}(x^i)\|^2\\\notag
&=|||(x^i-P_{Q_i}(x^i))_{i\in I}|||_w^2.
\end{align} 
Since, as is well known \cite[Theorem 3.16]{BauschkeCombettes2017book}, the minimum in the definition of the orthogonal projection is attained at a unique point, this point is $\mathbf{z}:=(P_{Q_i}(x^i))_{i\in I}$ (and obviously $\mathbf{z}\in \mathbf{Q}$). Hence $\mathbf{P}_{\mathbf{Q}}(\mathbf{x})=(P_{Q_i}(x^i))_{i\in I}$, namely \beqref{P (1.2.7)} holds.

Now we prove \beqref{P (1.2.8)}. Any point $\mathbf{z}\in \mathbf{D}$ satisfies $\mathbf{z}=\mathbf{J}(z)$ for some $z\in X$, and so 
\begin{multline}\label{eq:P_D_min}
|||\mathbf{x}-\mathbf{P}_{\mathbf{D}}(\mathbf{x})|||_w^2
=\min\{|||\mathbf{x}-\mathbf{z}|||_w^2: \mathbf{z}\in \mathbf{D}\}
=\min\{|||\mathbf{x}-\mathbf{J}(z)|||_w^2: z\in X\}\\
=\min\left\{\sum_{i\in I}w(i)\|x^i-z\|^2: z\in X\right\}.
\end{multline}
The minimum in the last expression is attained at the unique point $\wt{z}$ where the gradient of the function $g:X\to\R$, defined by $g(z):=\sum_{i\in I}w(i)\|x^i-z\|^2$, $z\in X$, vanishes. Since $\nabla g(z)=2\sum_{i\in I}w(i)(x^i-z)$ for each $z\in X$ and since $\sum_{i\in I}w(i)=1$, we conclude that $\wt{z}=\sum_{i\in I}w(i)x^i$. Hence \beqref{eq:P_D_min} implies that $\mathbf{P}_{\mathbf{D}}(\mathbf{x})=\mathbf{J}(\wt{z})=\mathbf{J}\left(\sum_{i\in I}w(i)x^i\right)$, as claimed. 
\end{proof}

We continue with the following definition and two propositions.

\begin{defin}
\label{P def 1.1.3} Let $S_{1}$ and $S_{2}$ be two subsets of a real Hilbert space $Y$. A hyperplane $\Psi $ is said to \texttt{strictly separate} $S_{1}\,$ \texttt{and} $\,S_{2}$ if $S_1$ is contained in the interior of one half-space induced by $\Psi$, and $S_2$ is contained in the interior of the opposite half-space. 
\end{defin}

\begin{prop}\label{P c 1.1.4}Let $T_{1}$ and $T_{2}$ be two cutters defined on a real Hilbert space $Y$. If there exists some $x\in \textnormal{Fix}(T_1T_2)\backslash \textnormal{Fix}(T_2)$, then $\textnormal{Fix}(T_{1})\cap \textnormal{Fix}(T_{2})=\emptyset$.
\end{prop}

\begin{proof}
Since $x\notin \textnormal{Fix}(T_{2})$, both $H(T_2(x),x)$ and $H(x,T_2(x))$ are half-spaces and  
\begin{equation}
\Psi:=\left\{ u\in Y\mid \left\langle u-\frac{1}{2}(T_{2}\left(
x\right) +x),\text{ }T_{2}\left( x\right) -x\right\rangle =0\right\}
\end{equation}
is a hyperplane. We claim that $\Psi$ strictly separates $H\left(T_{2}\left(x\right),x\right)$ and  $H\left( x,T_{2}\left(x\right)\right)$. Indeed, if, say, $u\in H(T_2(x),x)$, then $\langle u-x,T_2(x)-x\rangle\leq 0$ (see  \beqref{eq:halfspace}), and because $T_2(x)\neq x$, we have
\begin{align*}
&\left\langle u-\frac{1}{2}(T_2(x)+x), T_2(x)-x\right\rangle=\frac{1}{2}\left\langle u-T_2(x),T_2(x)-x\right\rangle+\frac{1}{2}\left\langle u-x,T_2(x)-x\right\rangle\\
&=\frac{1}{2}\left\langle u-x,T_2(x)-x\right\rangle+\frac{1}{2}\left\langle x-T_2(x),T_2(x)-x\right\rangle+\frac{1}{2}\left\langle u-x,T_2(x)-x\right\rangle\\
&=\langle u-x,T_2(x)-x\rangle-\frac{1}{2}\|x-T_2(x)\|^2<0+0=0.
\end{align*}
Hence, $u$ is in the interior of one of the half-spaces whose boundary is $\Psi$, that is, in the interior of $\{ v\in Y\mid \left\langle v-0.5(T_{2}\left(
x\right) +x),\text{ }T_{2}\left( x\right) -x\right\rangle \leq 0\}$. Similarly, if $u\in H(x,T_2(x))$, then $u$ is in the interior of the other half-space whose boundary is $\Psi$. Since $T_{1}$ and $T_{2}$ are cutters and since $T_1T_2(x)=x$, by denoting $z:=T_2(x)$ we obtain the following inclusions: 
\begin{equation*}
\textnormal{Fix}(T_{1})\subseteq H(z,T_1(z))=H\left(T_{2}(x),T_{1}T_{2}\left( x\right)
\right) =H\left(T_{2}\left(x\right),x\right),
\end{equation*}
\begin{equation*}
\textnormal{Fix}(T_{2})\subseteq H\left( x,T_{2}\left( x\right) \right).
\end{equation*}
Since we already know that the hyperplane $\Psi$ strictly  separates $H(x,T_2(x))$ and $H(T_2(x),x)$, we conclude that $\Psi$ strictly separates $\textnormal{Fix}(T_{1})$ and $\textnormal{Fix}(T_{2})$. Thus, one has $\textnormal{Fix}(T_{1})\cap \textnormal{Fix}(T_{2})=\emptyset$.
\end{proof}

\begin{prop}\label{prop:T_w neq x P_D(T(x)) neq x} %\label{P c 1.2.1/2} 
Let $\{T_{i}\}_{i\in I}$ be cutters and $\{F_{i}\}_{i\in I}$ be their fixed points sets, respectively, and let  $w:I\rightarrow\left( 0,1\right]$ be a positive weight function. If $\cap_{i\in I}F_{i}\neq \emptyset$ and $x\in X$ satisfies $x\notin \cap_{i\in I}F_{i}$, then $T_w(x)\neq x$ and $\mathbf{P}_{\mathbf{D}}(\mathbf{T}(\mathbf{x}))\neq \mathbf{x}$, where $\mathbf{x}:=\mathbf{J}(x)$. 
\end{prop}

\begin{proof}
Let $\mathbf{D}$ and $\mathbf{T}$ be as defined above in \beqref{P (1.2.1)} and \beqref{eq:T}, respectively, and let  $\mathbf{F}:=\textnormal{Fix}(\mathbf{T})$. By  \beqref{P (1.2.5)} and the obvious inclusion $\mathbf{x}=\mathbf{J}(x)\in \mathbf{D}$, we have $\mathbf{x}\notin \mathbf{F}$. Assume, for a contradiction, that $T_w(x)=x$. This equality implies  that $\mathbf{J}(T_w(x))=\mathbf{J}(x)=\mathbf{x}$. From this equality, \beqref{P (1.2.8)} and the equality $\mathbf{T}(\mathbf{x})=(T_i(x))_{i=1}^m$, we obtain $\mathbf{P}_{\mathbf{D}}(\mathbf{T}(\mathbf{x}))=\mathbf{J}(T_w(x))=\mathbf{J}(x)=\mathbf{x}$. The previous lines imply that we can apply Proposition \bref{P c 1.1.4} with $\mathbf{T}$ as $T_{2}$ there and $\mathbf{P}_{\mathbf{D}}$ as $T_{1}$ there. This proposition yields $\textnormal{Fix}(\mathbf{P}_{\mathbf{D}})\cap\mathbf{F}=\emptyset$, namely $\mathbf{D}\cap \mathbf{F}=\emptyset$. This result, together with \beqref{P (1.2.5)}, give $\cap_{i\in I}F_{i}=\emptyset$, a contradiction to the assumption that  $\cap_{i\in I}F_{i}\neq \emptyset $. Therefore,  $T_w(x)\neq x$.

It remains to show that $\mathbf{P}_{\mathbf{D}}(\mathbf{T}(\mathbf{x}))\neq \mathbf{x}$. Indeed, since $\mathbf{T}(\mathbf{x})=(T_i(x))_{i=1}^m$, it follows from \beqref{eq:T} and \beqref{P (1.2.8)} that $\mathbf{P}_{\mathbf{D}}(\mathbf{T}(\mathbf{x}))=\mathbf{J}(T_w(x))$. As a result, if  $\mathbf{P}_{\mathbf{D}}(\mathbf{T}(\mathbf{x}))=\mathbf{x}$, then $\mathbf{J}(T_w(x))=\mathbf{P}_{\mathbf{D}}(\mathbf{T}(\mathbf{x}))=\mathbf{x}=\mathbf{J}(x)$, and so, from the definition of $\mathbf{J}$, one has $T_w(x)=x$, a contradiction to what we showed in the previous paragraph. Hence, $\mathbf{P}_{\mathbf{D}}(\mathbf{T}(\mathbf{x}))\neq \mathbf{x}$.
\end{proof}
The following corollary is essentially known (see, for instance, \cite[Equation (19)]{Combettes2001inproc}; see \cite[Proposition 4.47, p. 85]{BauschkeCombettes2017book} for a more general result). Below we provide a new proof of it, based on Proposition \bref{prop:T_w neq x P_D(T(x)) neq x} above. 
\begin{cor}\label{cor:T_w(x)=x}
For each $i\in I$ suppose that $T_i:X\to X$ is a cutter having a fixed point set $F_i$ such that $\cap_{i\in I}F_i\neq\emptyset$, and suppose that $w:I\to [0,1]$ is a weight function. Then for all $x\in X$ one has  $T_{w}(x)=x$ if and only if $T_i(x)=x$ for all $i\in \wh{I}_w$. In other words, $\textnormal{Fix}(T_w)=\cap_{i\in \wh{I}_w}\textnormal{Fix}(T_i)$. 
\end{cor}
\begin{proof}
Fix an arbitrary $x\in X$. The triangle inequality, the definitions of $\wh{I}_w$ and $\wh{w}$, and the fact that $w$ and $\wh{w}$ are weight functions, all imply that  
\begin{equation}\label{eq:T_w <= sum|T_i(x)-x|}
0\leq \left\|\sum_{i\in I}w(i)T_i(x)-x\right\|=\left\|\sum_{i\in \wh{I}_w}\wh{w}(i)(T_i(x)-x)\right\|\leq \sum_{i\in \wh{I}_w}\wh{w}(i)\|T_i(x)-x\|.
\end{equation}
As a result of \beqref{eq:T_w <= sum|T_i(x)-x|}, if $T_i(x)=x$ for all $i\in \wh{I}_w$, then $\|\sum_{i\in I}w(i)T_i(x)-x\|=0$ and hence,  $T_w(x)=x$. As for the converse direction, suppose that $T_w(x)=x$, and assume, for a contradiction, that $T_i(x)\neq x$ for some $i\in \wh{I}_w$. Then $x\notin \cap_{i\in \wh{I}_w}F_i$. Thus, Proposition \bref{prop:T_w neq x P_D(T(x)) neq x}  (with $\wh{w}$ and $\wh{I}_w$ instead of $w$ and $I$ which appear there, respectively) implies that  $x\neq T_{\wh{w}}(x)$. Since $T_w=T_{\wh{w}}$, it follows that $x\neq T_w(x)$, a contradiction to our assumption. Thus, $T_i(x)=x$ for all $i\in \wh{I}_w$, as claimed. 
\end{proof}

Now suppose that $\cap_{i\in I}F_i\neq \emptyset$. We define $\mathbf{b}_{w}:\mathbf{D}\backslash\mathbf{F}\rightarrow \mathbf{D,}$ by 
\begin{equation}\label{eq:b_w} %\label{P (1.2.9)}
\mathbf{b}_{w}(\mathbf{x}):=\mathbf{x}+\frac{|||\mathbf{T}\left( \mathbf{x}
\right) -\mathbf{x}|||_w^{2}}{|||\mathbf{P}_{\mathbf{D}}\left( \mathbf{T}
\left( \mathbf{x}\right) \right) -\mathbf{x}|||_w^{2}}\left( \mathbf{P}_{
\mathbf{D}}\left( \mathbf{T}\left( \mathbf{x}\right) \right) -\mathbf{x}
\right), \quad \mathbf{x}\in \mathbf{D}\backslash\mathbf{F} 
\end{equation}
where $w$ is the positive weight function which appears in \beqref{eq:inner}. By  Proposition \bref{prop:T_w neq x P_D(T(x)) neq x} we have $\mathbf{P}_{\mathbf{D}}\left(\mathbf{T}\left(\mathbf{x}
\right)\right)\neq\mathbf{x}$ for all $\mathbf{x}\in \mathbf{D}\backslash \mathbf{F}$, a fact which guarantees that $\mathbf{b}_{w}(\mathbf{x})$ is well-defined. 

\begin{rem}\label{rem:P_D(P_C(x)) neq x}
Note that $\mathbf{b}_w$ (from \beqref{eq:b_w}) generalizes the expression $\mathbf{b}$ which appears in \cite[Lemma 1.2]{Pierra1984jour} with orthogonal projections instead of general cutters as we use here, and with constant weights ($w_k(i):=1/m$ for all $k\in\N\cup\{0\}$ and $i\in I$) and not general ones as we use here. In addition, and in Proposition \bref{prop:T_w neq x P_D(T(x)) neq x} above, we filled a gap which appears in \cite[Lemma 1.2]{Pierra1984jour} (see also \cite[Theorem 2.8(i) and Corollary 2.12(i)]{BauschkeCombettesKruk2006jour} and \cite[Proof of Proposition 2.4]{Combettes2001inproc} for closely related results): the gap there is the unproven assertion that, with the notation we use here, $\mathbf{P_D}(\mathbf{P_C}(\mathbf{x}))\neq \mathbf{x}$ for all $\mathbf{x}\in \mathbf{D}\backslash\mathbf{C}$, where $\mathbf{C}:=C_1\times\cdots\times C_m$ and $C_i$ is a nonempty, closed and convex subset of $X$ for all $i\in I$; while it is claimed in \cite[Lemma 1.2]{Pierra1984jour} that this assertion was proved in \cite[Theorem 1.1(i)]{Pierra1984jour},  as far as we understand, this is not the case. The reason that we indeed filled this gap follows from the  observation that if we denote $T_i:=P_{C_i}$ for each $i\in I$ and consider the mapping $\mathbf{T}$ from \beqref{eq:T}, then Lemma \bref{lem:P_Q P_D} (with $Q_i:=C_i$, $i\in I$, and hence $\mathbf{Q}=\mathbf{C}$) guarantees that $\mathbf{T}=\mathbf{P_C}$, and since both $F_i:=\textnormal{Fix}(T_i)=C_i$ for every $i\in I$ and $\mathbf{F}:=\textnormal{Fix}(\mathbf{T})=F_1\times\cdots \times F_m=\mathbf{C}$ (as follows from \beqref{eq:fixTF} and the equality $F_i=C_i$ for every $i\in I$), we conclude from  Proposition \bref{prop:T_w neq x P_D(T(x)) neq x} that $\mathbf{P}_{\mathbf{D}}\left(\mathbf{T}\left(\mathbf{x}\right)\right)\neq\mathbf{x}$, namely $\mathbf{P_D}(\mathbf{P_C}(\mathbf{x}))\neq \mathbf{x}$, as required. 
\end{rem}

The next lemma presents useful observations about $\mathbf{b}_{w}(\mathbf{x})$. It generalizes and extends \cite[Lemma 1.2]{Pierra1984jour}. 

\begin{lem}\label{lem:b_w}%\label{P c 1.2.2}
 Let $\{T_{i}\}_{i\in I}$ be cutters and $\{F_{i}\}_{i\in I}$ their fixed points sets, respectively, and suppose that $w:I\rightarrow\left(0,1\right]$ is a positive weight function. Let $\mathbf{D}$ and $\mathbf{T}$ be as defined above in \beqref{P (1.2.1)} and \beqref{eq:T}, respectively, and let  $\mathbf{F}:=\textnormal{Fix}(\mathbf{T})$. Assume that $x\in X$, $\mathbf{x}:=\mathbf{J}(x) \in \mathbf{D}\backslash \mathbf{F}$ and  $\cap_{i\in I}F_i\neq \emptyset$. Then the following assertions hold:

\begin{enumerate}[(i)]
\item \label{P c 1.2.2-1}$\mathbf{b}_{w}(\mathbf{x})$ is in the intersection of the ray 
\begin{equation}
\mathbf{L_x}:=\left\{ \mathbf{z}\in \mathbf{D}\mid \mathbf{z}=\mathbf{x}
+\lambda \left( \mathbf{P}_{\mathbf{D}}\left( \mathbf{T}\left( \mathbf{x}\right) \right) -\mathbf{x}\right) ,\text{ }\lambda \geq 0\right\}
\end{equation}
with the hyperplane 
\begin{equation}
\mathbf{\Psi :}=\left\{ \mathbf{z}\in X^{m}\mid
\left\langle \left\langle \mathbf{z}-\mathbf{T}\left( \mathbf{x}\right) ,\mathbf{T}\left( \mathbf{x}\right)-\mathbf{x}\right\rangle \right\rangle_w=0\right\}.
\end{equation}

\item \label{lambda_xw} If we denote 
\begin{equation}
\widehat{\lambda}_{x,w}:=\frac{|||\mathbf{T}\left( \mathbf{x}\right)-
\mathbf{x}|||_w^{2}}{|||\mathbf{P}_{\mathbf{D}}\left( \mathbf{T}\left( \mathbf{x}\right)\right)-\mathbf{x}|||_w^{2}},  \label{eq:lambdahat}
\end{equation}
then $\widehat{\lambda}_{x,w}\geq 1$ and 
\begin{equation}\label{eq:lambda_xw_components}
 \widehat{\lambda}_{x,w}=\frac{\sum_{i\in I}w(i)\left\|T_{i}\left(x\right)-x\right\|^2}{\left\|\sum_{i\in I}w(i)(T_{i}(x)-x)\right\|^{2}}.
\end{equation}

\item \label{P c 1.2.2-2}$\mathbf{D\cap F\subseteq D\cap H}\left( \mathbf{x},\mathbf{b}_{w}(\mathbf{x})\right) \subseteq\mathbf{D}\cap\mathbf{H}\left( 
\mathbf{x},\mathbf{P}_{\mathbf{D}}\left( \mathbf{T}\left( \mathbf{x}\right)
\right) \right)$.

\item\label{F_i in H(x,x+) in H(x,T_i)} The following inclusions hold:  
\begin{align}
\bigcap_{i\in I}F_{i} & \subseteq H\left( x,T_{w,\wh{\lambda}_{x,w}}(x)\right) \subseteq H\left( x,T_w(x)\right).
\label{eq:assert3}
\end{align}
\end{enumerate}
\end{lem}

\begin{proof}
We prove the assertions in the order in which they appeared.

\textbf{Proof of Part \beqref{P c 1.2.2-1}:} Because $\mathbf{P}_{\mathbf{D}}$ is an orthogonal projection onto the linear subspace $\mathbf{D}$, it follows from a basic property of orthogonal projections (see \cite[Corollary 3.22, p. 55]{BauschkeCombettes2017book}, with $z:=0$) that 
\begin{equation}
\left\langle \left\langle \mathbf{y},\mathbf{T}\left( \mathbf{x}\right) -\mathbf{P}_{
\mathbf{D}}\left( \mathbf{T}\left( \mathbf{x}\right) \right) \right\rangle
\right\rangle_w=0,\text{ for all }\mathbf{y}\in \mathbf{D}.
\label{P (1.2.10)}
\end{equation}
By the definition of $\mathbf{b}_{w}(\mathbf{x)}$ (see \beqref{eq:b_w}) and simple arithmetic, we have 
\begin{align}
& \left\langle \left\langle \mathbf{b}_{w}(\mathbf{x)-T}\left( \mathbf{x}%
\right) ,\mathbf{T}\left( \mathbf{x}\right) -\mathbf{x}\right\rangle
\right\rangle_w  \notag \\
& =\frac{|||\mathbf{T}\left( \mathbf{x}\right) -\mathbf{x}|||_w^{2}}{|||
\mathbf{P}_{\mathbf{D}}\left( \mathbf{T}\left( \mathbf{x}\right) \right)-
\mathbf{x}|||_w^{2}}\Big(
\left\langle \left\langle \mathbf{P}_{\mathbf{D}
}\left( \mathbf{T}\left( \mathbf{x}\right) \right) -\mathbf{x,\mathbf{T}
\left( \mathbf{x}\right) -P}_{\mathbf{D}}\left( \mathbf{T}\left( \mathbf{x}
\right) \right) \right\rangle \right\rangle_w  \notag \\
&  +\left\langle \left\langle \mathbf{P}_{\mathbf{D}}\left( \mathbf{T}
\left( \mathbf{x}\right) \right) -\mathbf{x,P}_{\mathbf{D}}\left( \mathbf{T}
\left( \mathbf{x}\right) \right) -\mathbf{x}\right\rangle \right\rangle_w
\Big)
+\left\langle \left\langle \mathbf{x-T}\left( \mathbf{x}\right),
\mathbf{T}\left( \mathbf{x}\right) -\mathbf{x}\right\rangle \right\rangle_w.
\end{align}
From this, along with \beqref{P (1.2.10)} and the fact that $\mathbf{P}_{
\mathbf{D}}\left( \mathbf{T}\left( \mathbf{x}\right) \right) -\mathbf{x}\in 
\mathbf{D,}$ we have 
\begin{align}
& \left\langle \left\langle \mathbf{b}_{w}(\mathbf{x)-T}\left( \mathbf{x}
\right) ,\mathbf{T}\left( \mathbf{x}\right) -\mathbf{x}\right\rangle
\right\rangle_w  \notag \\
& =\frac{|||\mathbf{T}\left( \mathbf{x}\right) -\mathbf{x}|||_w^{2}}{|||
\mathbf{P}_{\mathbf{D}}\left( \mathbf{T}\left( \mathbf{x}\right) \right)-
\mathbf{x}|||_w^{2}}|||\mathbf{P}_{\mathbf{D}}\left( \mathbf{T}\left( \mathbf{x}\right) \right) -\mathbf{x}|||_w^{2}-|||\mathbf{T}\left( \mathbf{x}\right)-\mathbf{x}|||_w^{2}=0.  \label{P (1.2.11)}
\end{align}
This shows that $\mathbf{b}_{w}(\mathbf{x)\in \Psi }$. Since $\mathbf{x}\notin \mathbf{F}$, we have $\mathbf{T}(\mathbf{x})\neq \mathbf{x}$ and $\wh{\lambda}_{x,w}>0$, and since $\mathbf{b}_{w}(\mathbf{x})=\mathbf{x}+\widehat{\lambda}_{x,w}(\mathbf{P}_{\mathbf{D}}\left( \mathbf{T}(\mathbf{x}))-\mathbf{x}\right)$, we have $\mathbf{b}_{w}(\mathbf{x})\in \mathbf{L_x}$. This completes the
proof of Part \beqref{P c 1.2.2-1}.

\textbf{Proof of Part \beqref{lambda_xw}: }
Proposition \bref{prop:T_w neq x P_D(T(x)) neq x} implies that $\mathbf{P}_{\mathbf{D}}\left(\mathbf{T}\left(\mathbf{x}
\right)\right)\neq\mathbf{x}$ for all $\mathbf{x}\in \mathbf{D}\backslash \mathbf{F}$ and so $\widehat{\lambda}_{x,w}$ is well defined. Since $\mathbf{x}\in \mathbf{D}$ and the orthogonal projection is nonexpansive (see, e.g.,  \cite[Theorem 2.2.21, p. 76]{Cegielski2012book}), we have 
\begin{equation*}
|||\mathbf{P}_{\mathbf{D}}(\mathbf{T}(\mathbf{x}))-\mathbf{x}|||_w=|||\mathbf{P}_{\mathbf{D}}(\mathbf{T}(\mathbf{x}))-\mathbf{P}_{\mathbf{D}}(\mathbf{x})|||_w\leq |||\mathbf{T}(\mathbf{x})-\mathbf{x}|||_w.
\end{equation*}
Thus, $\sqrt{\widehat{\lambda}_{x,w}}=|||\mathbf{T}(\mathbf{x})-\mathbf{x}|||_w/|||\mathbf{P}_{\mathbf{D}}(\mathbf{T}(\mathbf{x})-\mathbf{x}|||_w\geq 1$, and so $\widehat{\lambda}_{x,w}\geq 1$. 

Now we establish \beqref{eq:lambda_xw_components}. From the definition of the norm $|||\cdot|||_w$ it follows that 
\begin{equation}\label{eq:lambda_numerator}
|||\mathbf{T}(\mathbf{x})-\mathbf{x}|||_w^2=|||(T_i(x)-x)_{i\in I}|||_w^2=\sum_{i\in I}w(i)\|T_i(x)-x\|^2. 
\end{equation}

In addition, \beqref{P (1.2.8)}, the linearity of $\mathbf{P}_{\mathbf{D}}$, the fact that $\mathbf{x}=\mathbf{P}_{\mathbf{D}}(\mathbf{x})$ and the fact that $\sum_{j\in I}w(j)=1$, all of them imply that 
\begin{align}\label{eq:lambda_denominator}
&|||\mathbf{P}_{\mathbf{D}}(\mathbf{T}(\mathbf{x}))-\mathbf{x}|||_w^2\notag
=|||\mathbf{P}_{\mathbf{D}}(\mathbf{T}(\mathbf{x})-\mathbf{x})|||_w^2\notag\\
&=\left|\left|\left|\mathbf{J}\left(\sum_{i\in I}w(i)(T_i(x)-x)\right)\right|\right|\right|_w^2
=\sum_{j\in I}w(j)\left\|\sum_{i\in I}w(i)(T_i(x)-x)\right\|^2\notag\\
&=\left\|\sum_{i\in I}w(i)(T_i(x)-x)\right\|^2.
\end{align}
From \beqref{eq:lambda_numerator} and \beqref{eq:lambda_denominator} we obtain \beqref{eq:lambda_xw_components}.

\textbf{Proof of Part \beqref{P c 1.2.2-2}:} From the fact that $\mathbf{T}$ is a cutter, from \beqref{P (1.2.10)}, \beqref{P (1.2.11)} and from the fact that $\mathbf{b}_{w}(\mathbf{x)}\in \mathbf{D}$, it follows that for all $\mathbf{y}\in \mathbf{F}\cap \mathbf{D}$,  
\begin{align}
0& \leq \left\langle \left\langle \mathbf{y}-\mathbf{T}\left( \mathbf{x}%
\right) ,\mathbf{T}\left( \mathbf{x}\right) -\mathbf{x}\right\rangle
\right\rangle_w  \notag \\
& =\left\langle \left\langle \mathbf{y}-\mathbf{b}_{w}(\mathbf{x)},\mathbf{T}%
\left( \mathbf{x}\right) -\mathbf{x}\right\rangle \right\rangle_w
+\left\langle \left\langle \mathbf{b}_{w}(\mathbf{x)-\mathbf{T}\left( 
\mathbf{x}\right) ,T}\left( \mathbf{x}\right) -\mathbf{x}\right\rangle
\right\rangle_w  \notag \\
& =\left\langle \left\langle \mathbf{y}-\mathbf{b}_{w}(\mathbf{x)},\mathbf{T}%
\left( \mathbf{x}\right) -\mathbf{x}\right\rangle \right\rangle_w  \notag \\
& =\left\langle \left\langle \mathbf{y}-\mathbf{b}_{w}(\mathbf{x)},\mathbf{T}%
\left( \mathbf{x}\right) -\mathbf{P}_{\mathbf{D}}\left( \mathbf{T}\left( 
\mathbf{x}\right) \right) \right\rangle \right\rangle_w +\left\langle
\left\langle \mathbf{y}-\mathbf{b}_{w}(\mathbf{x)},\mathbf{P}_{\mathbf{D}%
}\left( \mathbf{T}\left( \mathbf{x}\right) \right) -\mathbf{x}\right\rangle
\right\rangle_w  \notag \\
& =\left\langle \left\langle \mathbf{y}-\mathbf{b}_{w}(\mathbf{x)},\mathbf{P}%
_{\mathbf{D}}\left( \mathbf{T}\left( \mathbf{x}\right) \right) -\mathbf{x}%
\right\rangle \right\rangle_w .  \label{P (1.2.13)}
\end{align}
By combining this inequality with \beqref{eq:b_w} and the assumption that $\mathbf{x}\notin \mathbf{F}$, one has 
\begin{equation*}
\left\langle \left\langle \mathbf{y}-\mathbf{b}_{w}(\mathbf{x)},\frac{|||\mathbf{P}_{\mathbf{D}}\left( \mathbf{T}\left( \mathbf{x}\right) \right) -\mathbf{x}|||_w^{2}}{|||\mathbf{T}\left( \mathbf{x}\right) -\mathbf{x}|||_w^{2}}\left( \mathbf{b}_{w}(\mathbf{x)-x}\right) \right\rangle \right\rangle_w \geq
0,\text{ for all }\mathbf{y}\in \mathbf{F}\cap \mathbf{D},
\end{equation*}
an inequality which implies that $\left\langle \left\langle \mathbf{y}-\mathbf{b}_{w}(\mathbf{x)
},\mathbf{b}_{w}(\mathbf{x)-x}\right\rangle \right\rangle_w \geq 0$, for all $
\mathbf{y}\in \mathbf{F}\cap \mathbf{D}$. This, in turn, implies that $\mathbf{F\cap D\subseteq D\cap H}\left( \mathbf{x},\mathbf{b}_{w}(\mathbf{x)}
\right)$.

Finally, it remains to show that 
$\mathbf{D\cap H}\left(\mathbf{x},\mathbf{b}_{w}(\mathbf{x)}\right)\subseteq \mathbf{D}\cap \mathbf{H}\left( \mathbf{x},\mathbf{P}_{\mathbf{D}}\left( \mathbf{T}\left( \mathbf{x}\right) \right) \right)$. Indeed, let $\mathbf{y}\in \mathbf{D\cap H}\left(\mathbf{x},\mathbf{b}_{w}(\mathbf{x)}\right)$. Then previous lines, as well as \beqref{eq:b_w} and \beqref{eq:lambdahat}, imply that 

\begin{align} 
0\leq \left\langle\left\langle \mathbf{y}-\mathbf{b}_w(\mathbf{x}), \mathbf{b}_w(\mathbf{x})-\mathbf{x}\right\rangle\right\rangle_w=\left\langle\left\langle \mathbf{y}-\mathbf{b}_w(\mathbf{x}),\widehat{\lambda}_{x,w}\left(\mathbf{P}_{\mathbf{D}}(\mathbf{T}(\mathbf{x}))-\mathbf{x}\right)\right\rangle\right\rangle_w.
\end{align}

Now we use this inequality and the fact that $\widehat{\lambda}_{x,w}=1+\epsilon$ for some $\epsilon\geq 0$ (as a result of Part \beqref{lambda_xw}), to conclude that 

\begin{align*} 
0&\leq \left\langle\left\langle \mathbf{y}-\mathbf{b}_w(\mathbf{x}),\mathbf{P}_{\mathbf{D}}(\mathbf{T}(\mathbf{x}))-\mathbf{x}\right\rangle\right\rangle_w\\
&=\left\langle\left\langle \mathbf{y}-(\mathbf{x}+\widehat{\lambda}_{x,w}\left(\mathbf{P}_{\mathbf{D}}(\mathbf{T}(\mathbf{x}))-\mathbf{x}\right)), \mathbf{P}_{\mathbf{D}}(\mathbf{T}(\mathbf{x}))-\mathbf{x})
\right\rangle\right\rangle_w\\
&=\left\langle\left\langle \mathbf{y}+\epsilon\mathbf{x}-(1+\epsilon)\mathbf{P}_{\mathbf{D}}(\mathbf{T}(\mathbf{x})),\mathbf{P}_{\mathbf{D}}(\mathbf{T}(\mathbf{x}))-\mathbf{x}\right\rangle\right\rangle_w\\
&=\left\langle\left\langle \mathbf{y}-\mathbf{P}_{\mathbf{D}}(\mathbf{T}(\mathbf{x})),\mathbf{P}_{\mathbf{D}}(\mathbf{T}(\mathbf{x}))-\mathbf{x}\right\rangle\right\rangle_w+\epsilon\left\langle\left\langle\mathbf{x}-\mathbf{P}_{\mathbf{D}}(\mathbf{T}(\mathbf{x})),\mathbf{P}_{\mathbf{D}}(\mathbf{T}(\mathbf{x}))-\mathbf{x}\right\rangle\right\rangle_w.
\end{align*}

Therefore, 
\begin{equation}
0\leq \epsilon|||\mathbf{x}-\mathbf{P}_{\mathbf{D}}(\mathbf{T}(\mathbf{x}))|||^2_w\leq  \left\langle\left\langle \mathbf{y}-\mathbf{P}_{\mathbf{D}}(\mathbf{T}(\mathbf{x})),\mathbf{P}_{\mathbf{D}}(\mathbf{T}(\mathbf{x}))-\mathbf{x}\right\rangle\right\rangle_w,
\end{equation}
namely $\mathbf{y}\in \mathbf{D}\cap \mathbf{H}\left(\mathbf{x},\mathbf{P}_{\mathbf{D} 
}\left( \mathbf{T}\left( \mathbf{x}\right) \right) \right)$  (see \beqref{eq:halfspace}), as required. 

\textbf{Proof of Part \beqref{F_i in H(x,x+) in H(x,T_i)}:}  
The definitions of $\mathbf{b}_{w}(\mathbf{x)}$, $\mathbf{T}$ and $\mathbf{D}$, together with Lemma \bref{lem:P_Q P_D} and the linearity of $\mathbf{J}$, imply that
\begin{equation}\label{eq:bw=J}
\mathbf{b}_{w}(\mathbf{x)}=\mathbf{J}\left(T_{w,\widehat{\lambda}_{x,w}}(x)\right).
\end{equation}
Given $q\in\cap_{i\in I}F_i$, it follows from \beqref{P (1.2.5)} that $\mathbf{J}(q)\in \mathbf{F\cap D}$. Hence, \beqref{eq:bw=J}, Part \beqref{P c 1.2.2-2} and the definition of $\mathbf{J}$ imply that $\mathbf{J}(q)\in \mathbf{D}\cap\mathbf{H}(\mathbf{J}(x),\mathbf{J}(T_{w,\wh{\lambda}_{x,w}}(x)))$. Thus, the definitions of $\mathbf{J}$ and the inner product $\langle\langle\cdot,\cdot\rangle\rangle_w$ imply that $q\in H(x,T_{w,\wh{\lambda}_{x,w}}(x))$. Since $q$ was an arbitrary point in $\cap_{i\in I}F_i$, we have $\cap_{i\in I}F_i\subseteq H(x,T_{w,\wh{\lambda}_{x,w}}(x))$.

It remains to prove the second inclusion in Part \beqref{F_i in H(x,x+) in H(x,T_i)}. In order to show this, let $u\in  H(x,T_{w,\wh{\lambda}_{x,w}}(x))$ be arbitrary. This fact, as well as \beqref{eq:halfspace} and the fact that $\wh{\lambda}_{x,w}=1+\epsilon$ for some $\epsilon\geq 0$ (see Part \beqref{lambda_xw}), imply that
\begin{align*}
0&\leq \langle u-T_{w,\wh{\lambda}_{x,w}}(x),T_{w,\wh{\lambda}_{x,w}}(x)-x\rangle=
\left\langle u-T_{w,\wh{\lambda}_{x,w}}(x),\wh{\lambda}_{x,w}\left(T_{w}(x)-x\right)\right\rangle\\
&=\left\langle u-\left(x+(1+\epsilon)\left(T_w(x)-x\right)\right),\wh{\lambda}_{x,w}\left(T_{w}(x)-x\right)\right\rangle\\
&=\left\langle u-T_w(x),\wh{\lambda}_{x,w}\left(T_{w}(x)-x\right)\right\rangle
-\epsilon\left\langle T_w(x)-x,\wh{\lambda}_{x,w}\left(T_{w}(x)-x\right)\right\rangle.
\end{align*}
Consequently, 
\begin{equation*}
0\leq\epsilon\left\|T_w(x)-x\right\|^2\leq\left\langle u-T_w(x),T_w(x)-x\right\rangle,
\end{equation*}
and hence, $u\in H(x,T_w(x))$, as required.
\end{proof}

The following proposition essentially appears in \cite[Corollary 3.4(i)]{BargetzKolobovReichZalas2019jour}, with a different notation and with a somewhat terse proof. We provide a new proof of it below. 

\begin{prop}
\label{P c 2.1.6}
For each $i\in I$ suppose that $T_i:X\to X$ is a cutter having a fixed point set $F_i$. Let $\tau_{1}$ and $\tau _{2}$ be in $(0,1]$, let $x\in X$ and let $w$ be a weight function. If $\lambda \in \left[ \tau _{1},\left( 2-\tau_{2}\right) L(x,w)\right] $ and $q\in F:=\cap_{i\in I}F_i$, then the following inequality is satisfied: 
\begin{equation}
\left\Vert T_{w,\lambda }\left( x\right) -q\right\Vert ^{2}\leq \left\Vert
x-q\right\Vert ^{2}-\tau _{1}\tau _{2}\sum_{i\in I}w\left( i\right)
\left\Vert T_{i}\left( x\right) -x\right\Vert ^{2}.  \label{eq:prop11}
\end{equation}
\end{prop}

\begin{proof}
Observe first that since $L(x,w)\geq 1>0$ and since $\lambda\in [\tau_1,(2-\tau_2)L(x,w)]$, we have $2-\tau_2\geq \lambda/L(x,w)$ and $\lambda\geq\tau_1$. Thus,   \begin{equation}\label{eq:lambda_L(x,w)}
\lambda\left(2-\frac{\lambda}{L(x,w)}\right)\geq \tau_1\tau_2.
\end{equation}

If $x\in F$, then $T_{i}\left( x\right) -x=0$ for all $i\in I$. Thus $\ T_{w,\lambda }\left( x\right) =x$ and hence \beqref{eq:prop11} is clear. Suppose now that  $x\notin F$. We observe that $T_w(x)=T_{\wh{w}}(x)$, where $\wh{w}:\wh{I}_w\to [0,1]$ is the restriction of $w$ to $\wh{I}_w$ (recall that $\wh{I}_w:=\{i\in I\mid w\left( i\right)>0\}$, and, therefore, $\wh{w}$ is a positive weight function defined on $\wh{I}_w$). By using Proposition \bref{prop:T_w neq x P_D(T(x)) neq x} with $\wh{w}$ instead of $w$ (in order to use this proposition we also need to verify that $\cap_{i\in \wh{I}_w}F_i\neq \emptyset$, which is true because $q\in F\subseteq\cap_{i\in \wh{I}_w}F_i$), we obtain that $T_{\wh{w}}(x)\neq x$, and hence  $T_{w}(x)-x\neq 0$. Hence, it follows from \beqref{eq:L(x,w)} and simple calculations that  
\begin{align}
\left\Vert T_{w,\lambda }\left( x\right) -q\right\Vert ^{2}& =\left\Vert
x+\lambda \left( T_{w}\left( x\right) -x\right) -q\right\Vert ^{2}  \notag \\
& =\left\Vert x-q\right\Vert ^{2}+2\lambda \left\langle T_{w}\left( x\right)
-x,x-q\right\rangle +\lambda ^{2}\left\Vert T_{w}\left( x\right)
-x\right\Vert ^{2}  \notag \\
& =\left\Vert x-q\right\Vert ^{2}+2\lambda \sum_{i\in I}w\left( i\right)
\left\langle T_{i}\left( x\right) -x,x-q\right\rangle  \notag \\
& +\frac{\lambda ^{2}}{L\left( x,w\right) }\sum_{i\in I}w\left( i\right)
\left\Vert T_{i}\left( x\right) -x\right\Vert ^{2}.
\end{align}
By adding and subtracting $T_{i}\left( x\right) $ in the inner product of the second summand on the right-hand side of the last equality, we get 
\begin{align}
\left\Vert T_{w,\lambda }\left( x\right) -q\right\Vert ^{2}& =\left\Vert
x-q\right\Vert ^{2}+2\lambda \sum_{i\in I}w\left( i\right) \left\langle
T_{i}\left( x\right) -x,x-T_{i}\left( x\right) \right\rangle  \notag \\
& +2\lambda \sum_{i\in I}w\left( i\right) \left\langle T_{i}\left( x\right)
-x,T_{i}\left( x\right) -q\right\rangle  \notag \\
& +\frac{\lambda ^{2}}{L\left( x,w\right) }\sum_{i\in I}w\left( i\right)
\left\Vert T_{i}\left( x\right) -x\right\Vert ^{2}.
\end{align}%
By this result, the assumption that the operators $T_{i},$ $i\in I$ are cutters (and hence they obey \beqref{eq:directop}) and by \beqref{eq:lambda_L(x,w)}, we get the desired result:
\begin{align}
\left\Vert T_{w,\lambda }\left( x\right) -q\right\Vert ^{2}& \leq \left\Vert
x-q\right\Vert ^{2}+2\lambda \sum_{i\in I}w\left( i\right) \left\langle
T_{i}\left( x\right) -x,x-T_{i}\left( x\right) \right\rangle  \notag \\
& +\frac{\lambda ^{2}}{L\left( x,w\right) }\sum_{i\in I}w\left( i\right)
\left\Vert T_{i}\left( x\right) -x\right\Vert ^{2}  \notag \\
& =\left\Vert x-q\right\Vert ^{2}-\lambda \left( 2-\frac{\lambda }{L\left(
x,w\right) }\right) \sum_{i\in I}w\left( i\right) \left\Vert T_{i}\left(
x\right) -x\right\Vert ^{2}  \notag \\
& \leq \left\Vert x-q\right\Vert ^{2}-\tau _{1}\tau _{2}\sum_{i\in I}w\left(
i\right) \left\Vert T_{i}\left( x\right) -x\right\Vert ^{2}.
\end{align}%
\end{proof}

The next corollary generalizes Corollary \bref{cor:T_w(x)=x}, and also improves upon \cite[Corollary 2.12]{BauschkeCombettesKruk2006jour} and \cite[Proposition 2.4]{Combettes2001inproc} (assuming the index set $I$ in \cite[Proposition 2.4]{Combettes2001inproc} is finite) in the sense that $\lambda$ can be larger than $L(x,w)$.
  
\begin{cor}\label{cor:T_i(x)=x}
For each $i\in I$ suppose that $T_i:X\to X$ is a cutter having a fixed point set $F_i$ such that $\cap_{i\in I}F_i\neq\emptyset$. Let $\tau _{1}$ and $\tau _{2}$ be in $(0,1]$, let $x\in X$ and suppose that  $w:I\to [0,1]$ is a weight function. If $\lambda \in \left[ \tau _{1},\left( 2-\tau_{2}\right) L(x,w)\right]$ and if $T_{w,\lambda}(x)=x$, then   $T_i(x)=x$ for all $i\in \wh{I}_w$. 
\end{cor}
\begin{proof}
Let $q\in F$. Since $T_{w,\lambda}(x)=x$, we conclude from Proposition \bref{P c 2.1.6} that 
\begin{multline*}
0\leq\sum_{i\in \wh{I}_w}w(i)\|T_i(x)-x\|^2=\sum_{i\in I}w(i)\|T_i(x)-x\|^2\leq \frac{\|x-q\|^2-\|T_{w,\lambda}(x)-q\|^2}{\tau_1\tau_2}=0.
\end{multline*}
Since $w(i)>0$ for all $i\in \wh{I}_w$, the sum $\sum_{i\in \wh{I}_w}w(i)\|T_i(x)-x\|^2$ can vanish if and only if $\|T_i(x)-x\|^2=0$ for each $i\in \wh{I}_w$. Thus, $T_i(x)=x$ for all $i\in \wh{I}_w$, as claimed.
\end{proof}

The next lemmata are used for deriving a certain variant of Proposition \bref{P c 2.1.6}, namely Proposition \bref{prop:FejerT_{w,lambda}} (these propositions do not seem to imply each other).

\begin{lem}\label{lem:H(T_lambda_12)}
Suppose that $T: X\to X$ is an operator (not necessarily a cutter), $x\in X$ is arbitrary, $\lambda_2>0$ and $\lambda_1\in [0,\lambda_2]$. Then $H(x,T_{\lambda_2}(x))\subseteq H(x,T_{\lambda_1}(x))$. Moreover, if $x\notin \textnormal{Fix}(T)$ and $0<\lambda_1<\lambda_2$, then $H(x,T_{\lambda_2}(x))\varsubsetneqq H(x,T_{\lambda_1}(x))$.
\end{lem}
\begin{proof}
We start by showing that $H(x,T_{\lambda_2}(x))\subseteq H(x,T_{\lambda_1}(x))$. Let $z\in H(x,T_{\lambda_2}(x))$ be arbitrary. Then \beqref{eq:halfspace} implies that $\langle z-T_{\lambda_2}(x),x-T_{\lambda_2}(x)\rangle\leq 0$. This inequality, as well as simple calculations  and the facts that $0\leq\lambda_1\leq \lambda_2$ and $0<\lambda_2$, show that 
\begin{multline}\label{eq:H(T_lambda_12)}
\langle z-T_{\lambda_1}(x),x-T_{\lambda_1}(x)\rangle
=\langle z-(x+\lambda_1(T(x)-x)),x-(x+\lambda_1(T(x)-x))\rangle\\
=\langle z-\left(x+\lambda_2(T(x)-x)+(\lambda_1-\lambda_2)(T(x)-x)\right),-\lambda_1(T(x)-x)\rangle\\
=\langle z-(x+\lambda_2(T(x)-x)),-\lambda_1(T(x)-x)\rangle
-\langle (\lambda_1-\lambda_2)(T(x)-x),-\lambda_1(T(x)-x)\rangle\\
=(\lambda_1/\lambda_2)\langle z-(x+\lambda_2(T(x)-x)),-\lambda_2(T(x)-x)\rangle+\lambda_1(\lambda_1-\lambda_2)\|T(x)-x\|^2\\
=(\lambda_1/\lambda_2)\left\langle z-T_{\lambda_2}(x),x-T_{\lambda_2}(x)\right\rangle+\lambda_1(\lambda_1-\lambda_2)\|T(x)-x\|^2
\leq 0+0=0.
\end{multline}
Hence, $z\in H(x,T_{\lambda_1}(x))$, and since $z$ was an arbitrary point in $H(x,T_{\lambda_2}(x))$, one has  $H(x,T_{\lambda_2}(x))\subseteq H(x,T_{\lambda_1}(x))$, as required. 

Finally, if $x\notin \textnormal{Fix}(T)$ and $0<\lambda_1<\lambda_2$, then $\lambda_1(\lambda_1-\lambda_2)\|T(x)-x\|^2<0$. Thus, any $z\in H(x,T_{\lambda_2}(x))$ satisfies \beqref{eq:H(T_lambda_12)} with a strict inequality in the last line of \beqref{eq:H(T_lambda_12)}. This fact, as well as the fact that $z:=T_{\lambda_1}(x)$ obviously satisfies $\langle z-T_{\lambda_1}(x),x-T_{\lambda_1}(x)\rangle=0$, imply that this specific $z$ is in $H(x,T_{\lambda_1}(x))$ and it cannot be in $H(x,T_{\lambda_2}(x))$.  Hence,  $H(x,T_{\lambda_2}(x))\varsubsetneqq H(x,T_{\lambda_1}(x))$.
\end{proof}

\begin{lem}\label{lem:T(w,wh(lambda))}
Let $\{T_{i}\}_{i\in I}$ be cutters and $\{F_{i}\}_{i\in I}$ their fixed points sets, respectively, such that $F:=\cap_{i\in I}F_i\neq\emptyset$. Let $w:I\to [0,1]$ be a  weight function. Fix some $x\in X$ and suppose that $\lambda\in [0,\wh{\lambda}_{x,\wh{w}}]$, where $\wh{\lambda}_{x,\wh{w}}$ satisfies  \beqref{eq:lambda_xw_components} (and in Lemma \bref{lem:b_w} we use, instead of $I$ and $w$, the subset $\wh{I}_w$ and the  positive weight function $\wh{w}:\wh{I}_w\to(0,1]$, respectively). Then $F\subseteq H(x,T_{w,\lambda}(x))$.
\end{lem}

\begin{proof}
Since $\wh{\lambda}_{x,\wh{w}}>0$ according to Lemma \bref{lem:b_w}\beqref{lambda_xw}, it follows from Lemma \bref{lem:H(T_lambda_12)}, with $T_{w}$ instead of the operator $T$ used there, that 
\begin{equation}\label{eq:S<=T_w,lambda}
H(x,T_{w,\widehat{\lambda}_{x,\wh{w}}}(x))\subseteq H(x,T_{w,\lambda}(x)).
\end{equation}
In addition, from Part \beqref{F_i in H(x,x+) in H(x,T_i)} of Lemma \bref{lem:b_w} (the inclusion of the first set in the second one in \beqref{eq:assert3}, where in Lemma \bref{lem:b_w} we use $\wh{I}_w$ and $\wh{w}$ instead of $I$ and $w$, respectively), we have $\cap_{i\in \wh{I}_w}F_i\subseteq H(x,T_{w,\widehat{\lambda}_{x,\wh{w}}}(x))$. Since $\wh{I}_w\subseteq I$, we have $F=\cap_{i\in I}F_i\subseteq\cap_{i\in \wh{I}_w}F_i$. These inclusions, as well as \beqref{eq:S<=T_w,lambda}, imply that 
\begin{equation*}
F\subseteq \bigcap_{i\in \wh{I}_w}F_i\subseteq H(x,T_{w,\widehat{\lambda}_{x,\wh{w}}}(x))\subseteq H(x,T_{w,\lambda}(x)).
\end{equation*}
\end{proof}

\begin{prop}\label{prop:FejerT_{w,lambda}}%\label{P c 2.1.6.A}
Let $\{T_{i}\}_{i\in I}$ be cutters and $\{F_{i}\}_{i\in I}$ their fixed points sets, respectively, such that $F:=\cap_{i\in I}F_i\neq\emptyset$. Suppose that $\tau _{1}$ and $\tau _{2}$ are in $(0,1]$. Let $x\in X$ and let $w:I\to [0,1]$ be a weight function. If $\lambda \in \left[ \tau _{1},\left( 2-\tau_{2}\right) L(x,w)\right] $ and $q\in F$, then the following inequality is satisfied: 
\begin{equation}
\left\| T_{w,\lambda }\left( x\right) -q\right\| ^{2}\leq \left\|
x-q\right\| ^{2}-\tau_2 \left\| T_{w,\lambda }\left( x\right) -x\right\| ^{2}.
\end{equation}
\end{prop}

\begin{proof}
From  \beqref{eq:L(x,w)}, \beqref{eq:lambda_xw_components}, and the definitions of $\wh{I}_w$ and $\wh{w}$, we have 
\begin{equation}\label{eq:L}
L\left( x,w\right) =\widehat{\lambda }_{x,\wh{w}}\text{ or }L\left( x,w\right) =1.
\end{equation}
Since $\widehat{\lambda }_{x,\wh{w}}\geq 1$ according to Part \beqref{lambda_xw} of Lemma \bref{lem:b_w} (in which we use $\wh{I}_w$ and $\wh{w}$ instead of $I$ and $w$, respectively), it follows from \beqref{eq:L} that $L(x,w)\in [1,\widehat{\lambda }_{x,\wh{w}}]$.

Now we divide the proof into two cases, depending on the value of $\lambda$. 

\vspace*{0.2cm}
\textbf{Case 1:} $\tau _{1}\leq \lambda \leq L\left( x,w\right)$. 

In this case $\lambda\leq L(x,w)\leq \wh{\lambda}_{x,\wh{w}}$, and so Lemma \bref{lem:T(w,wh(lambda))} implies that $F\subseteq H(x,T_{w,\lambda}(x))$. Thus,    
\begin{equation}
\left\langle T_{w,\lambda }(x)-x,T_{w,\lambda }(x)-q\right\rangle \leq 0.  \label{P (2.1.5/2)}
\end{equation}
This fact, simple calculations and the fact that $\tau_2\in (0,1]$, all imply that  
\begin{align}\label{P (2.1.5/3)}
&\left\Vert T_{w,\lambda }\left( x\right) -q\right\Vert ^{2}
=\|(T_{w,\lambda}(x)-x)+(x-q)\|^2\\\notag
&=\left\Vert x-q\right\Vert ^{2}+2\left\langle T_{w,\lambda }\left( x\right)-x,x-q\right\rangle +\left\Vert T_{w,\lambda }\left( x\right) -x\right\Vert^{2}  \\\notag
& =\left\Vert x-q\right\Vert ^{2}+2\left\langle T_{w,\lambda }\left(
x\right) -x,x-T_{w,\lambda }\left( x\right) \right\rangle  \\\notag
&+2\left\langle T_{w,\lambda }\left( x\right) -x,T_{w,\lambda }\left(
x\right) -q\right\rangle +\left\Vert T_{w,\lambda }\left( x\right)
-x\right\Vert ^{2}  \\\notag
 &\leq \left\Vert x-q\right\Vert ^{2}-2\|T_{w,\lambda}(x)-x\|^2 
 +\left\Vert T_{w,\lambda }\left( x\right) -x\right\Vert ^{2} \\\notag
 &=\left\Vert x-q\right\Vert ^{2}-\left\Vert T_{w,\lambda }\left( x\right)-x\right\Vert ^{2}\leq \|x-q\|^2-\tau_2\|T_{w,\lambda}(x)-x\|^2.  
\end{align}

\textbf{Case 2:} $L\left( x,w\right) \leq\lambda\leq\left( 2-\tau_{2}\right) L\left( x,w\right)$. 

In this case, if we let $\alpha:=\lambda/L(x,w)$, then simple calculations show that $1\leq \alpha\leq 2-\tau_2$ and 
\begin{equation}\label{eq:Tw-lambda=}
T_{w,\lambda}\left( x\right) =x+\alpha\left( T_{w,L\left( x,w\right) }\left(x\right) -x\right).
\end{equation}
Therefore,    
\begin{align}\label{eq:T_w_lambda_q}
&\left\| T_{w,\lambda}\left( x\right) -q\right\| ^{2}  =\left\|
x+\alpha\left( T_{w,L\left( x,w\right) }\left( x\right) -x\right) -q\right\|^{2} \notag\\
 &=\left\| x-q\right\| ^{2}+2\alpha\left\langle T_{w,L\left( x,w\right)
}\left( x\right) -x,x-q\right\rangle+\alpha^{2}\left\| T_{w,L\left( x,w\right) }\left( x\right) -x\right\| ^{2} \notag\\
 &=\left\| x-q\right\|^{2}+2\alpha\left\langle T_{w,L\left( x,w\right)
}\left( x\right) -x,x-T_{w,L\left( x,w\right) }\left( x\right) \right\rangle\notag\\
&+2\alpha\left\langle T_{w,L\left( x,w\right) }\left( x\right)-x,T_{w,L\left( x,w\right) }\left( x\right) -q\right\rangle   
 +\alpha^{2}\left\| T_{w,L\left( x,w\right) }\left( x\right) -x\right\|
^{2}.
\end{align}
Since $L(x,w)\leq\widehat{\lambda}_{x,\wh{w}}$, it follows from  Lemma \bref{lem:T(w,wh(lambda))} that $F\subseteq H(x,T_{w,L(x,w)}(x))$, and hence, $\langle T_{w,L\left( x,w\right) }\left( x\right)
-x,T_{w,L\left( x,w\right) }\left( x\right) -q\rangle\leq 0$.  This fact, as well as \beqref{eq:Tw-lambda=}, \beqref{eq:T_w_lambda_q}, and the inequality $\tau_2\leq 2-\alpha$, imply that   
\begin{align*}
&\left\| T_{w,\lambda}\left( x\right) -q\right\| ^{2}\\  &\leq\left\|
x-q\right\|^{2}+2\alpha\left\langle T_{w,L\left( x,w\right) }\left(x\right) -x,x-T_{w,L\left( x,w\right) }\left( x\right) \right\rangle+\alpha^{2}\left\| T_{w,L\left( x,w\right) }\left( x\right) -x\right\| ^{2}\\
&=\left\|x-q\right\| ^{2}-\alpha\left( 2-\alpha\right) \left\| 
T_{w,L\left( x,w\right) }\left( x\right) -x\right\| ^{2}\\ &\leq\left\| x-q\right\| ^{2}-\tau_{2}\left\| \alpha\left( T_{w,L\left(x,w\right) }\left( x\right) -x\right) \right\| ^{2}\\
&=\left\| x-q\right\|^{2}-\tau_{2}\left\| T_{w,\lambda}\left( x\right)-x\right\|^{2}.  
\end{align*}
\end{proof}

\section{The block iterative extrapolated algorithm and the convergence theorems}\label{sec:AlgorithmConvergence}

In this section we present the extrapolated block-iterative algorithm aimed at solving the common fixed point problem \beqref{eq:CFPP}, and show its convergence under certain conditions. In a nutshell, the  algorithm generates the iteration $x^{k+1}$ by considering a block $I_k\subseteq I$, calculating the convex combination, for the block $I_k$,  of the differences $T_i(x^k)-x^k$, $i\in I_k$, and implementing an extrapolation in order to reach a \textit{deep step} towards the common fixed point set $\cap_{i\in I}F_{i}$. In each iterative step the user can choose weights and extrapolation parameters anew, as long as they obey some reasonable conditions.

\begin{algorithm} \textbf{(The extrapolated block-iterative algorithm).}\label{alg:Extrapolation}%\label{P alg 2.1.5}\\

{\noindent \textbf{Input:}} 
A real Hilbert space $X$, a positive integer $m$, an index set $I:=\{1,2,\ldots,m\}$, an arbitrary initialization point $x^{0}\in X$, two real numbers $\tau_{1}$ and $\tau_{2}$ in the interval $(0,1]$, a family of cutters $\{T_{i}\}_{i\in I}$ defined on $X$ with fixed point sets  $F_{i}:=\textnormal{Fix}(T_{i})=\{x\in X\mid T_{i}(x)=x\}$
and a nonempty common fixed point set $F:=\cap_{i\in I}F_{i}$, a sequence $\{w_{k}\}_{k=0}^{\infty}$ of weight functions with respect to $I$, and a sequence of relaxation parameters $\{\lambda_{k}\}_{k=0}^{\infty}$ which satisfy $\lambda_k\in \left[ \tau
_{1},\left( 2-\tau _{2}\right) L\left( x^{k},w_{k}\right) \right]$ for each $k\in \N\cup\{0\}$, where $L$ is from \beqref{eq:L(x,w)}.\\

{\noindent \textbf{Iterative step:}} Given the current iterate $x^{k}$, $k\in \N\cup\{0\}$,  calculate the
next iterate by 
\begin{equation}\label{eq:x^{k+1}}
x^{k+1}:=T_{w_{k},\lambda _{k}}\left( x^{k}\right)=x^k+\lambda_k\left(\sum_{i\in I}w_k(i)T_i(x^k)-x^k\right),
\end{equation}
that is, if we denote by $I_k:=\wh{I}_{w_k}=\{i\in I\,|\,w_k(i)>0\}$ the $k$-th block, then  
\begin{equation}\label{eq:x^{k+1}_block}
x^{k+1}=x^k+\lambda_k\sum_{i\in I_k}w_{k}(i)(T_i(x^k)-x^k).
\end{equation}
\end{algorithm}

\begin{rem}
Since $\tau_1/(2-\tau_2)\leq 1/(2-1)\leq L(x^k,w_k)$ for all $k\in\N\cup\{0\}$ according to \beqref{eq:L(x,w)}, it follows that for all $k\in\N\cup\{0\}$ the interval $[\tau_1,(2-\tau_2)L(x^k,w_k)]$ is nonempty, and hence one can indeed choose a sequence of relaxation parameters $\{\lambda_{k}\}_{k=0}^{\infty}$ in this interval. In standard BIP algorithms the parameters $\{\lambda_{k}\}_{k=0}^{\infty}$ are confined to the interval $\left[\tau_{1},2-\tau_{2}\right]$ for any user-chosen two real numbers $\tau_{1}$ and $\tau_{2}$ in the interval $(0,1],$ allowing under-relaxation or over-relaxation of the projections. Many papers investigate the role of relaxation parameters, see, e.g., \cite{ElfvingHansenNikazad2012jour}. Here, in Algorithm \bref{alg:Extrapolation}, the term $L(x^{k},w_{k})$ which might be  greater than $1$, enables \textquotedblleft deeper\textquotedblright\ steps which constitute the extrapolation beyond the usual interval that appears in convergence theorems for projection methods.
\end{rem}

\begin{rem}
Algorithm \bref{alg:Extrapolation} becomes fully sequential whenever at each iteration $k$ there is an index $j_k\in I$ such that $w_k(i)=0$ for all $i\in I\backslash\{j_k\}$ and $w_k(j_k)=1$, and it becomes fully simultaneous whenever $w_k(i)>0$ for all $k\in \N\cup\{0\}$ and all $i\in I$. Hence Algorithm \bref{alg:Extrapolation} can be used in both serial and parallel computational settings.
\end{rem}

 In order to ensure the convergence of Algorithm \bref{alg:Extrapolation}, we will impose, in either Theorem \bref{thm:NonemotyInterior} or Theorem   \bref{thm:AlmostCyclic} below, one of the following two conditions on the sequence of weights $\{w_k\}_{k=0}^{\infty}$, respectively: 
\begin{condition}\label{cond:sum_infty}
For all $i\in I$ one has $\sum_{k=0}^{\infty}w_k(i)=\infty$.
\end{condition}
\begin{condition}\label{cond:sw}
There are $s\in\N$ and $\alpha>0$ such that for all $i\in I$ and all $\ell\in\N\cup\{0\}$ there is some $k\in\{\ell,\ell+1,\ldots,\ell+s-1\}$ such that $w_k(i)\geq \alpha$.
\end{condition}

\begin{rem}\label{rem:ComparisonConditions}
Condition \bref{cond:sum_infty}, which essentially appeared first in \cite[p. 171]{AharoniCensor1989jour}, is strictly more general than Condition \bref{cond:sw}. Indeed, if $\{w_k\}_{k=0}^{\infty}$ is a sequence of weight functions which satisfies Condition \bref{cond:sw}, then, in particular, for all $i\in I$ and all $p\in\N\cup\{0\}$ there is some $t_p\in\{ps,ps+1,\ldots,ps+p-1\}$ such that $w_{t_p}(i)\geq \alpha$, and therefore, $\sum_{k=0}^{\infty}w_k(i)\geq \sum_{p=0}^{\infty}w_{t_p}(i)\geq \sum_{p=0}^{\infty}\alpha=\infty$, namely, Condition \bref{cond:sw} implies Condition \bref{cond:sum_infty}. On the other hand, the converse is not true since if, for instance, for every $k\in\N\cup\{0\}$ we let $w_k(i):=1/((k+1)m)$ for each $i\in I\backslash\{m\}$ and $w_k(m):=1-(m-1)/((k+1)m)$, then $\{w_k\}_{k=0}^{\infty}$ is a sequence of weight functions which satisfies Condition \bref{cond:sum_infty} but does not satisfy Condition \bref{cond:sw}.

However, Condition \bref{cond:sw}, which seems to be new, is strictly more general than the condition on the weights imposed in either \cite[Corollary 4.2]{BargetzKolobovReichZalas2019jour} (up to a typo in \cite[Relation (104)]{BargetzKolobovReichZalas2019jour}, that the weights are constant and not dynamic: in later lines there the weights are assumed to be dynamic) or \cite[Step 1 (St.1) on page 14232]{BuongAnh2023jour}, when in both cases we restrict ourselves  to strings of length 1. This condition is the following one: 
\begin{multline}\label{eq:IntermittentPositive}
\textnormal{There are } s\in\N\,\, \textnormal{and}\, \alpha>0\,\, \textnormal{such that both}\,\, \{I_k\}_{k=0}^{\infty}\,\, \textnormal{is}\,\, s\textnormal{-intermittent}\\
\textnormal{and}\, w_k(i)\geq\alpha\,\, \textnormal{for all}\,\, k\in\N\cup\{0\}\,\,\textnormal{and all}\,\, i\in I_k,  
\end{multline}
where by saying that the sequence of blocks $\{I_k\}_{k=0}^{\infty}$ is $s$-intermittent we mean that $I=I_{\ell}\cup I_{\ell+1}\cup\ldots\cup I_{\ell+s-1}$ for each $\ell\in\N\cup\{0\}$ (it seems that  \beqref{eq:IntermittentPositive} appeared first in \cite[Theorem 5.1]{Cegielski2015jour}, although variants of it can be found in previous works such as \cite[Definition 1.2 Part (c) and Definition 3.1 Relation (3.4)]{Combettes1997jour-b}). Indeed, suppose that  \beqref{eq:IntermittentPositive} holds, and fix $\ell\in\N\cup\{0\}$ and $i\in I$. Since the sequence of blocks is $s$-intermittent and $i\in I$, we have $i\in I=I_{\ell}\cup I_{\ell+1}\cup\ldots\cup I_{\ell+s-1}$, and hence there is some $k\in \{\ell,\ell+1,\ldots,\ell+s-1\}$ such that $i\in I_k$. Thus, \beqref{eq:IntermittentPositive} implies that $w_k(i)\geq \alpha$, and we conclude that Condition \bref{cond:sw} holds (with the same $s$ and $\alpha$ as in \beqref{eq:IntermittentPositive}). On the other hand, there are cases where Condition \bref{cond:sw} holds but \beqref{eq:IntermittentPositive} does not: see Examples \bref{ex:1/(2m)}--\bref{ex:1/m} below.

Another condition on the weights appears in \cite[Condition 5.4]{BauschkeBorweinCombettes2003jour}, \cite[Algorithm 3.1(4) and Condition 3.2(ii)]{BauschkeCombettesKruk2006jour}, and \cite[Algorithm 6.1(4) and Definition 6.3]{Combettes2001inproc} (see also \cite[Definition 3.1 and Algorithm 6.3 (C2)]{Combettes2000jour} \cite[Algorithm 3.9 and Theorem 3.11]{CombettesWoodstock2021jour}). When restricted to the case where $I$ is finite ($I$ in \cite[Algorithm 5.1]{BauschkeBorweinCombettes2003jour}, \cite[Algorithm 3.1]{BauschkeCombettesKruk2006jour} and \cite[Algorithm 6.1]{Combettes2001inproc} is allowed to be countable) it essentially says that the sequence of blocks is $s$-intermittent for some $s\in \N$ and also that there is a positive number $\delta\in (0,1)$ having the following property: for all $k\in\N\cup\{0\}$ there is an index $j_{\textnormal{max},k}\in I_k$ for which the maximum $\max\{g_j\,|\, j\in I_k\}$ is attained (where for each $j\in I_k$ one has that $g_j$ is a certain nonnegative number depending on $j$, $x^k$, and the given cutters) and this $j_{\textnormal{max},k}$ satisfies $w_k(j_{\textnormal{max},k})\geq \delta$. This condition neither implies Condition \bref{cond:sw} nor is implied by it. Indeed, if we consider the sequence $\{w_k\}_{k=0}^{\infty}$ of weight functions defined in Example \bref{ex:1/m} below, then Condition \bref{cond:sw} holds, but for the above mentioned condition to hold it must be that $j_{\textnormal{max},k}=m$ (because for each $j\neq m$ one has $\lim_{k\to\infty,k\neq 0\,(\!\!\!\!\mod m)}w_k(j)=0$), which is usually not true. On the other hand, in Condition \bref{cond:sw} we impose a requirement which should be satisfied by all  the indices in $I$ regarding the uniform positive lower bound on the corresponding weights at the indices, while in \cite[Condition 5.4]{BauschkeBorweinCombettes2003jour}, \cite[Algorithm 3.1(4) and Condition 3.2(ii)]{BauschkeCombettesKruk2006jour}, and \cite[Algorithm 6.1(4) and Definition 6.3]{Combettes2001inproc} this is not the case.

Finally, we note that the concept of intermittent controls was introduced in \cite[Definition 3.18]{BauschkeBorwein1996jour}, following the notion of almost cyclic controls which seems to appear first in \cite{Lent1976inproc} (even though a more general control appeared before in \cite[Definition 5]{Browder1967jour}; this latter condition seems to inspire the generalized intermittency conditions which appear in \cite[Condition 5.4(iii)]{BauschkeBorweinCombettes2003jour}, \cite[Condition 3.2(ii)]{BauschkeCombettesKruk2006jour}, \cite[Definition 3.1]{Combettes2000jour} and \cite[Definition 6.3]{Combettes2001inproc}). 
\end{rem}

\begin{expl}\label{ex:1/(2m)}
Suppose that $m>1$. Let $s$ be an arbitrary even natural number and let $s':=s/2$. For each $t\in\N\cup\{0\}$ choose randomly, say using the uniform distribution on $\{ts',ts'+1,\ldots,ts'+s'-1\}$, a number $h_{t,1}\in \{ts',ts'+1,\ldots,ts'+s'-1\}$. Now choose (possibly randomly, using the uniform distribution on $[1/(2m),1/m]$) arbitrary real numbers  $w_{h_{t,1}}(i)\in [1/(2m),1/m]$ for all $i\in I\backslash\{m\}$, and define $w_{h_{t,1}}(m):=1-\sum_{j=1}^{m-1}w_{h_{t,1}}(j)$. Now for each $k\in\{ts',ts'+1,\ldots,ts'+s'-1\}\backslash\{h_{t,1}\}$ and each $i\in I\backslash\{m\}$ choose an arbitrary (possibly randomly, using the uniform distribution on $[0,1/m]$) real number $w_k(i)\in [0,1/m]$, and define $w_k(m):=1-\sum_{j=1}^{m-1}w_k(j)$. 

By doing this we obtain a sequence $\{w_k\}_{k=0}^{\infty}$ of weight functions which satisfies Condition \bref{cond:sw} with $\alpha:=1/(2m)$ and $s$.  Indeed, for every $\ell\in\N\cup\{0\}$ let $t:=\lceil \ell/s'\rceil$ (where $\lceil \cdot\rceil$ is the strict ceiling function, which assigns to every $r\in \R$ the minimal integer which is strictly greater than $r$), and let $k:=h_{t,1}$. The definition of $h_{t,1}$ implies that $k\in  \{ts',ts'+1,\ldots,ts'+s'-1\}$. In addition, the definition of $t$ implies that $(t-1)s'\leq\ell<ts'$. Hence, 
\begin{equation*}
0<ts'-\ell\leq k-\ell\leq ts'+s'-1-(t-1)s'= 2s'-1=s-1,
\end{equation*} 
and so $h_{t,1}=k\in\{\ell,\ell+1,\ldots,\ell+s-1\}$. The definition of $h_{t,1}$ implies that $w_{h_{t,1}}(i)\in [1/(2m),1/m]$ if $i\in I\backslash\{m\}$ and 
\begin{equation*}
w_{h_{t,1}}(m)=1-\sum_{j=1}^{m-1}w_{h_{t,1}}(j)\geq 1-\frac{m-1}{2m}=\frac{1}{2}+\frac{1}{2m}>\frac{1}{2m}=\alpha.
\end{equation*}
It follows that for all $i\in I$ and all $\ell\in\N\cup\{0\}$ there is $k\in\{\ell,\ell+1,\ldots,\ell+s-1\}$ (namely, $k:=h_{t,1}$ for $t:=\lceil \ell/s'\rceil$) such that $w_k(i)\geq \alpha$, that is, Condition \bref{cond:sw} does hold with the above mentioned $s$ and $\alpha$, as claimed. 

Finally, one observes that since $w_k(i)$ can be an arbitrary number  in $[0,1/m]$ whenever $k\neq h_{t,1}$ and $i\neq m$, there are cases in the choice of the weight functions where they are all positive and therefore all the blocks $I_k$ coincide with $I$ (and hence the sequence $\{I_k\}_{k=0}^{\infty}$ of blocks is 1-intermittent) and $\inf\{w_k(i)\,|\, k\in \N\cup\{0\}\}=0$ for all $i\in I\backslash\{m\}$. In these cases \beqref{eq:IntermittentPositive} does not hold. 
\end{expl}

\begin{expl}\label{ex:1/m}
Suppose that $m>1$. Given $k\in\N\cup\{0\}$ and $i\in I\backslash\{m\}$, define $w_k(i):=1/m$ whenever $k=0\, (\!\!\!\!\mod m)$ and $w_k(i):=1/(2km)$ otherwise. In addition, define $w_k(m):=1/m$ if  $k=0 \, (\!\!\!\!\mod m)$ and $w_k(m):=1-(m-1)/(2km)$ otherwise. Then $\{w_k\}_{k=0}^{\infty}$ becomes a sequence of positive weight functions which satisfies Condition \bref{cond:sw} with $s:=m$ and $\alpha:=1/m$, the $k$-th block is $I_k=I$ and hence the sequence $\{I_k\}_{k=0}^{\infty}$ of blocks is 1-intermittent, but \beqref{eq:IntermittentPositive} does not hold because $\inf\{w_k(i)\,|\, k\in \N\cup\{0\}\}=0$ for all $i\in I\backslash\{m\}$. 
\end{expl}

\begin{defin} A sequence $\{y^k\}_{k=0}^{\infty}$ in $X$ is said to be  Fej\'{e}r-monotone with respect to some subset $\emptyset\neq S\subseteq X$ if the following condition holds: $\|y^{k+1}-z\|\leq \|y^k-z\|$ for every $k\in\N\cup\{0\}$ and every $z\in S$. 
\end{defin}

\begin{prop}\label{prop:FejerMonotone} Any sequence $\left\{ x^{k}\right\}_{k=0}^{\infty}$, generated by Algorithm \bref{alg:Extrapolation}, is Fej\'{e}r-monotone with respect to the (assumed nonempty) common fixed point set $F$.
\end{prop}

\begin{proof}
By Proposition \bref{P c 2.1.6}, for any sequence $\left\{ x^{k}\right\}_{k=0}^{\infty }$, generated by Algorithm \bref{alg:Extrapolation}, and any $q\in F$, we have 
\begin{equation}
\left\Vert x^{k+1}-q\right\Vert ^{2}\leq \left\Vert x^{k}-q\right\Vert
^{2}-\tau _{1}\tau _{2}\sum_{i\in I}w_{k}\left( i\right) \left\Vert
T_{i}\left( x^{k}\right) -x^{k}\right\Vert ^{2}.
\end{equation}
This implies that $\|x^{k+1}-q\|\leq\|x^k-q\|$ for all $k\in\N\cup\{0\}$, and hence $\left\{x^{k}\right\}_{k=0}^{\infty}$ is Fej\'{e}r-monotone with respect to $F$. 
\end{proof}

\begin{prop}\label{P c 2.1.8/2}  Suppose that $X$ is finite-dimensional, that Condition \bref{cond:sum_infty} holds, that all the cutters $T_i$ are continuous, and that $\left\{ x^{k}\right\}_{k=0}^{\infty }$ is a sequence generated by Algorithm \bref{alg:Extrapolation}. If this sequence converges to some point $x^{\ast}\in X$, then $x^{\ast }\in F$.
\end{prop}

\begin{proof}
Assume, for a contradiction, that $x^{\ast }\notin F=\cap_{i\in I}F_i$, namely that  there is some index $i_{0}\in I$ such that $x^{\ast }\notin F_{i_{0}}$. This fact and the fact that $F_i$ is closed for each $i\in I$ (and, in particular, for $i_0$), implies the existence of some  $\varepsilon>0$ such that $B[x^{\ast},\varepsilon]\cap F_{i_{0}}=\emptyset$. The continuity of $T_{i_{0}}$ on $B[x^{\ast},\varepsilon]$, as well as the compactness of $B[x^{\ast},\varepsilon]$, the Weierstrass Maximal Value Theorem and the fact that $\|T_{i_0}(x)-x\|>0$ for all $x\in B[x^{\ast},\varepsilon]$, all imply that there is some $\delta>0$ such that $\left\Vert T_{i_{0}}\left( x\right)-x\right\Vert \geq \delta$ for all $x\in B[x^{\ast},\varepsilon]$. Since $x^{\ast }=\lim_{k\to\infty}x^k$, there is some $k_0\in\N\cup\{0\}$ such that $x^{k}\in B[x^{\ast},\varepsilon]$ for all $k>k_0$. Thus, Proposition \bref{P c 2.1.6} yields  
\begin{equation}\label{eq:i_0 delta}
\left\Vert x^{k+1}-q\right\Vert ^{2} \leq
\left\Vert x^{k}-q\right\Vert^{2}-\tau _{1}\tau _{2}\delta^2 w_{k}(i_{0})
\end{equation}
for all $k>k_0$ and all $q\in F$. Hence, for all $k_0+1\leq\ell\in\N$
\begin{equation*}
\sum_{k=k_0+1}^{\ell}w_k(i_0)\leq \frac{1}{\tau_{1}\tau _{2}\delta^2}\sum_{k=k_0+1}^{\ell}\left(\|x^k-q\|^2-\|x^{k+1}-q\|^2\right)\leq \frac{1}{\tau_{1}\tau _{2}\delta^2}\|x^{k_0+1}-q\|^2.
\end{equation*}
By letting $\ell\to\infty$ we conclude that $\sum_{k=k_0+1}^{\infty}w_k(i_0)$ is bounded from above, and this is impossible because we assumed that $\sum_{k=0}^{\infty}w_k(i)=\infty$ for every $i\in I$ (Condition \bref{cond:sum_infty}) and hence $\sum_{k=k_0+1}^{\infty}w_k(i_0)=\infty$. Therefore, the assumption that $x^{\ast}\notin F$ is impossible, and hence $x^{\ast}\in F$, as required.
\end{proof}

The following proposition brings together several well-known results, and is used below for proving the convergence of our algorithmic sequence. 
\begin{prop}\label{prop:FejerProperties}
Suppose that $\{y^k\}_{k=0}^{\infty}$ is a Fej\'er monotone sequence in $X$ with respect to some nonempty subset $C$. Then: 
\begin{enumerate}[(i)]
\item\label{item:Bounded} $\{y^k\}_{k=0}^{\infty}$ is bounded.
\item\label{item:WeakClusterPoint} $\{y^k\}_{k=0}^{\infty}$ has at least one weak sequential cluster  point.
\item\label{item:|x^k-q|=|u-q|} For every weak sequential cluster point $u$ of $\{y^k\}_{k=0}^{\infty}$ and every $q\in C$ one has $\|u-q\|\leq \|y^k-q\|$ for all $k\in\N\cup\{0\}$ and $\lim_{k\to\infty}\|y^k-q\|=\|u-q\|$. 
\item\label{item:Cluster==>Converges} If every weak sequential  cluster point of $\{y^k\}_{k=0}^{\infty}$ belongs to $C$, then $\{y^k\}_{k=0}^{\infty}$ converges weakly to a point in $C$. 
\item\label{item:Interior==>Converges} If the interior of $C$ is nonempty, then $\{y^k\}_{k=0}^{\infty}$ converges strongly to some point in $X$. 
\item\label{item:FiniteDimensional} If $X$ is finite-dimensional, then the words ``weak'' and ``weakly'' in Parts \beqref{item:WeakClusterPoint}--\beqref{item:Cluster==>Converges}  above can be replaced by the words ``strong'' and ``strongly'', respectively. 
\end{enumerate}
\end{prop}
\begin{proof}
Let $q\in C$ be arbitrary. Since $\{y^k\}_{k=0}^{\infty}$ is Fej\'er monotone, one has $\|y^{k+1}-q\|\leq\|y^k-q\|\leq\ldots\leq\|y^0-q\|$ for all $k\in\N\cup\{0\}$. Thus, $\{y^k\}_{k=0}^{\infty}$ is in the ball $B[q,\|y^0-q\|]$. Thus, Part \beqref{item:Bounded} holds. Part \beqref{item:WeakClusterPoint} follows immediately from Part \beqref{item:Bounded} since any bounded sequence in a Hilbert space has a weakly convergent subsequence \cite[Lemma 2.45]{BauschkeCombettes2017book}.  Since $\{\|y^k-q\|\}_{k=0}^{\infty}$ is monotone decreasing for all $q\in F$ from the definition of Fej\'er monotonicity, $\lim_{k\to\infty}\|y^k-q\|$ exists and satisfies $\lim_{k\to\infty}\|y^k-q\|\leq \|y^t-q\|$ for every $t\in \N\cup\{0\}$, and since the norm is weakly sequentially lower semicontinuous \cite[II.3.27, p. 68]{DunfordSchwartz1958book}, one has $\|u-q\|\leq \liminf_{k\to\infty}\|y^k-q\|=\lim_{k\to\infty}\|y^k-q\|$ whenever $u$ is a weak sequential cluster point of $\{y^k\}_{k=0}^{\infty}$. Thus, Part \beqref{item:|x^k-q|=|u-q|} holds. For the proof of Part  \beqref{item:Cluster==>Converges}, see \cite[Theorem 5.5, p. 92]{BauschkeCombettes2017book} or \cite[Corollary 3.3.3, p. 110]{Cegielski2012book}. For the proof of Part \beqref{item:Interior==>Converges}, see  \cite[Proposition 5.10, p. 94]{BauschkeCombettes2017book}. Part \beqref{item:FiniteDimensional} is immediate since in finite-dimensional spaces a sequence converges weakly if and only if it converges strongly \cite[Lemma 2.51(ii), p. 39]{BauschkeCombettes2017book}. 

Finally, we note that the proofs of Part \beqref{item:Cluster==>Converges} and  Part \beqref{item:Interior==>Converges}, respectively, can be found also in \cite[Theorem 2.16(ii)]{BauschkeBorwein1996jour} and \cite[Theorem 2.16(iii)]{BauschkeBorwein1996jour},  respectively, and while the assertions there are formulated under the assumption that $C$ is  nonempty, closed and convex, the proofs actually hold if merely $C\neq\emptyset$. 
\end{proof}

\begin{prop}\label{prop:x^{k+1}-x^k}
For any sequence $\left\{ x^{k}\right\} _{k=0}^{\infty }$, which is generated by Algorithm \bref{alg:Extrapolation}, we have:
\begin{enumerate}[(i)]
\item\label{item:sum|x^{k+1}-x^k|^2} $\sum_{k=0}^{\infty}\|x^{k+1}-x^k\|^2<\infty$,
\item\label{item:|x^{k+1}-x^{k}|=0} $\lim_{k\rightarrow \infty }\left\Vert x^{k+1}-x^{k}\right\Vert=0$.
\item\label{item:|x^{h_t}-x^{k_t}|=0} $\lim_{t\to\infty}\|x^{h_t}-x^{k_t}\|=0$ whenever $\{h_t\}_{t=0}^{\infty}$ and $\{k_t\}_{t=0}^{\infty}$ are subsequences of natural numbers which have the following properties: $k_t<h_t$ for every $t\in\N\cup\{0\}$ and $\sup\{h_t-k_t\,|\, t\in\N\cup\{0\}\}<\infty$. 
\end{enumerate}
\end{prop}
\begin{proof}
We start with Part \beqref{item:sum|x^{k+1}-x^k|^2}. 
Take any $q\in F$. Algorithm \bref{alg:Extrapolation} and  Proposition \bref{prop:FejerT_{w,lambda}} imply that for all $k\in\N\cup\{0\}$,
\begin{equation*}
\left\| x^{k+1}-q\right\| ^{2}\leq \left\| x^{k}-q\right\|^{2}-\tau_2\left\|x^{k+1}-x^{k}\right\|^{2}.
\end{equation*}
Hence, for all $\ell\in\N$, 
\begin{align*}
&\sum_{k=0}^{\ell}\left\| x^{k+1}-x^{k}\right\| ^{2}\leq \frac{1}{\tau_2}\sum_{k=0}^{\ell}\left(\|x^k-q\|^2-\|x^{k+1}-q\|^2\right)\notag\\
&=\frac{1}{\tau_2}\left(\|x^{0}-q\|^2-\|x^{\ell+1}-q\|^2\right)\leq \frac{1}{\tau_2}\|x^0-q\|^2.
\end{align*}
By letting $\ell\to\infty$, we have $\sum_{k=0}^{\infty}\|x^{k+1}-x^k\|^2\leq \|x^0-q\|^2/\tau_2<\infty$, as claimed. 

As for Part \beqref{item:|x^{k+1}-x^{k}|=0}, it follows from Part \beqref{item:sum|x^{k+1}-x^k|^2} and well-known results regarding nonnegative series that $\lim_{k\rightarrow \infty }\left\Vert x^{k+1}-x^{k}\right\Vert^2=0$, and hence $\lim_{k\rightarrow \infty }\left\Vert x^{k+1}-x^{k}\right\Vert=0$. 

It remains to prove Part \beqref{item:|x^{h_t}-x^{k_t}|=0}. Let $s:=\sup\{h_t-k_t\,|\, t\in\N\cup\{0\}\}$. By our assumptions on the subsequences $\{h_t\}_{t=0}^{\infty}$ and $\{k_t\}_{t=0}^{\infty}$, it follows that $s$ is a natural number. Denote $\beta_{\ell,t}:=\|x^{k_t+\ell}-x^{k_t+\ell-1}\|$ for every $\ell\in \{1,2,\ldots,s\}$ and every $t\in\N\cup\{0\}$. Then each of the $s$ sequences $\{\beta_{\ell,t}\}_{t=0}^{\infty}$, $\ell\in\{1,2,\ldots,s\}$ is a subsequence of the sequence $\{\|x^{k+1}-x^k\|\}_{k=0}^{\infty}$, and hence, as follows from Part \beqref{item:|x^{k+1}-x^{k}|=0}, we have $\lim_{t\to\infty}\beta_{\ell,t}=0$. Since the triangle inequality and the definition of $s$ imply that $\|x^{h_t}-x^{k_t}\|\leq \sum_{\ell=1}^{h_t-k_t}\|x^{k_t+\ell}-x^{k_t+\ell-1}\|\leq \sum_{\ell=1}^s\beta_{\ell,t}$ for all $t\in\N\cup\{0\}$, it follows from previous lines that $\lim_{t\to\infty}\|x^{h_t}-x^{k_t}\|=0$. 
\end{proof}

\begin{prop}\label{prop:AccumulationFixedPoint_cond_sw}
Suppose that $X$ is finite-dimensional, that $\{x^k\}_{k=0}^{\infty}$ is generated by Algorithm \bref{alg:Extrapolation}, that all the cutters $T_i$, $i\in I$ are continuous, and that Condition \bref{cond:sw} holds. If $x_{\infty}$ is an accumulation point of $\{x^k\}_{k=0}^{\infty}$, then $x_{\infty}\in F$.
\end{prop}
\begin{proof}
Since $x_{\infty}$ is an accumulation point of $\{x^k\}_{k=0}^{\infty}$, there is a subsequence $\{x^{k_t}\}_{t=0}^{\infty}$ such that $x_{\infty}=\lim_{t\to\infty}x^{k_t}$. We need to show that $x_{\infty}\in F_i$ for all $i\in I$. Fix an arbitrary $i\in I$. By Condition \bref{cond:sw}, for all $t\in\N\cup\{0\}$ there is $h_{t,i}\in \{k_t,k_t+1,\ldots,k_t+s-1\}$ such that $w_{h_{t,i}}(i)\geq \alpha$. Proposition \bref{prop:x^{k+1}-x^k} ensures that $\lim_{t\to\infty}\|x^{h_{t,i}}-x^{k_t}\|=0$, and hence $\|x^{h_{t,i}}-x_{\infty}\|\leq \|x^{h_{t,i}}-x^{k_t}\|+\|x^{k_t}-x_{\infty}\|\xrightarrow[t\to\infty]{} 0+0=0$. Consider the sequence $\{(w_{h_{t,i}}(1),w_{h_{t,i}}(2),\ldots,w_{h_{t,i}}(m))\}_{t=0}^{\infty}$ of weight vectors in $[0,1]^m$. Since $[0,1]^m$ is a compact subset of $\R^m$, there is an infinite  subset $N_1$ of $\N\cup\{0\}$ and a subsequence $\{(w_{h_{t,i}}(1),w_{h_{t,i}}(2),\ldots,w_{h_{t,i}}(m))\}_{t\in N_1}$ of $\{(w_{h_{t,i}}(1),w_{h_{t,i}}(2),\ldots,w_{h_{t,i}}(m))\}_{t=0}^{\infty}$ which converges to some weight vector $(w_{\infty,i}(1),w_{\infty,i}(2),\ldots,w_{\infty,i}(m))$. This vector satisfies $w_{\infty,i}(i)=\lim_{t\to\infty, t\in N_1}w_{h_{t,i}}(i)\geq \alpha$ because $w_{h_{t,i}}(i)\geq \alpha$ for all $t\in\N\cup\{0\}$ by the choice of $h_{t,i}$. 

Now there are two possibilities: either the sequence $\{L(x^{h_{t,i}},w_{h_{t,i}})\}_{t\in N_1}$ is bounded, or it is unbounded. Consider first the case where this sequence is bounded. Since $\lambda_k\in [\tau_1,(2-\tau_2)L(x^k,w_k)]$ for all $k\in\N\cup\{0\}$, we conclude that 
$\{\lambda_{h_{t,i}}\}_{t\in N_1}$ is a bounded sequence of positive numbers. Hence, there is a an infinite subset $N_2$ of $N_1$ and a subsequence $\{\lambda_{h_{t,i}}\}_{t\in N_2}$ of $\{\lambda_{h_{t,i}}\}_{t\in N_1}$ which converges to some $\lambda_{\infty,i}$ as $t\to\infty, t\in N_2$. Since $\{\lambda_{h_{t,i}}\}_{t\in N_2}$ is also a subsequence of $\{\lambda_{k}\}_{k\in \N\cup\{0\}}$ and since $\lambda_k\geq \tau_1$ for every $k\in\N\cup\{0\}$, we have $\lambda_{\infty,i}\geq \tau_1>0$. Since \beqref{eq:x^{k+1}} implies that 
\begin{equation*}
x^{h_{t,i}+1}=T_{w_{h_{t,i}},\lambda_{h_{t,i}}}(x^{h_{t,i}})\\
=x^{h_{t,i}}+\lambda_{h_{t,i}}\sum_{j\in I}w_{h_{t,i}}(j)(T_j(x^{h_{t,i}})-x^{h_{t,i}})
\end{equation*}
and since $T_j$ is continuous for all $j\in I$, we have 
\begin{align*}
\lim_{t\to\infty, t\in N_2}x^{h_{t,i}+1}&=x_{\infty}+\lambda_{\infty,i}\sum_{j\in I}w_{\infty,i}(j)(T_j(x_{\infty})-x_{\infty})\\
&=x_{\infty}+\lambda_{\infty,i}(T_{w_{\infty,i}}(x_{\infty})-x_{\infty}).
\end{align*}
On the other hand, $\lim_{t\to\infty,t\in N_2}\|x^{h_{t,i}+1}-x^{h_{t,i}}\|=0$ by Proposition \bref{prop:x^{k+1}-x^k}, and hence one has $\lim_{t\to\infty,t\in N_2}x^{h_{t,i}+1}=\lim_{t\to\infty,t\in N_2}(x^{h_{t,i}+1}-x^{h_{t,i}})+\lim_{t\to\infty,t\in N_2}x^{h_{t,i}}=x_{\infty}$. We conclude that $x_{\infty}=x_{\infty}+\lambda_{\infty,i}(T_{w_{\infty,i}}(x_{\infty})-x_{\infty})$, and therefore, using the fact that $\lambda_{\infty,i}>0$, we have $x_{\infty}=T_{w_{\infty,i}}(x_{\infty})$. But $w_{\infty,i}(i)\geq \alpha>0$ as explained earlier, namely $i\in \wh{I}_{w_{\infty,i}}$. Hence, Corollary \bref{cor:T_w(x)=x} implies that $T_i(x_{\infty})=x_{\infty}$, namely $x_{\infty}\in F_i$.

It remains to consider the case where $\{L(x^{h_{t,i}},w_{h_{t,i}})\}_{t\in N_1}$ is unbounded. In this case, since $L(x^{h_{t,i}},w_{h_{t,i}})\geq 1>0$ for all $t\in N_1$, the unboundedness is from above, and,  therefore, there is an infinite set $N_2$ of $N_1$ and a subsequence $\{L(x^{h_{t,i}},w_{h_{t,i}})\}_{t\in N_2}$ of $\{L(x^{h_{t,i}},w_{h_{t,i}})\}_{t\in N_1}$ for which $\lim_{t\to\infty, t\in N_2}L(x^{h_{t,i}},w_{h_{t,i}})=\infty$. Thus,  $L(x^{h_{t,i}},w_{h_{t,i}})>1$ for all sufficiently large $t\in N_2$, and hence \beqref{eq:L(x,w)} implies the equality  $L(x^{h_{t,i}},w_{h_{t,i}})=\sum_{j\in I}w_{h_{t,i}}(j)\|T_{j}(x^{h_{t,i}})-x^{h_{t,i}}\|^{2}/\|T_{w_{h_{t,i}}}(x^{h_{t,i}})-x^{h_{t,i}}\|^{2}$ for all sufficiently large $t\in N_2$. Since $\{x^{k_t}\}_{t=0}^{\infty}$ converges, it is bounded (boundedness also follows from the Fej\'er monotonicity of $\{x^k\}_{k=0}^{\infty}$), and hence it is contained in some closed ball $B[0,\rho]$ for some $\rho>0$. 

Since the $m$ operators $T_j$, $j\in I$ are continuous, they are bounded on the compact ball $B[0,\rho]$ by the Weierstrass Extreme Value Theorem. Since $w_k(j)\in [0,1]$ for all $k\in\N\cup\{0\}$ and all $j\in I$, the previous lines imply that $\sup\{\sum_{j\in I}w_{h_{t,i}}(j)\|T_{j}(x^{h_{t,i}})-x^{h_{t,i}}\|^{2}\,|\,t\in N_2\}<\infty$. Consequently, the equality $\lim_{t\to\infty, t\in N_2}L(x^{h_{t,i}},w_{h_{t,i}})=\infty$ can hold only if $\lim_{t\to\infty, t\in N_2}\|T_{w_{h_{t,i}}}(x^{h_{t,i}})-x^{h_{t,i}}\|=0$ (since otherwise there is some $\epsilon>0$ such that $\|T_{w_{h_{t,i}}}(x^{h_{t,i}})-x^{h_{t,i}}\|\geq \epsilon$ for every $t$ which belongs to some infinite subset $N_3$ of $N_2$, and then $\sup\{L(x^{h_{t,i}},w_{h_{t,i}})\,|\,t\in N_3\}$ is bounded by $\sup\{\sum_{j\in I}w_{h_{t,i}}(j)\|T_{j}(x^{h_{t,i}})-x^{h_{t,i}}\|^{2}\,|\,t\in N_2\}/\epsilon^2<\infty$, and this contradicts the equality $\lim_{t\to\infty, t\in N_3}L(x^{h_{t,i}},w_{h_{t,i}})=\lim_{t\to\infty, t\in N_2}L(x^{h_{t,i}},w_{h_{t,i}})=\infty$). 

Since $\lim_{t\to\infty, t\in N_2}x^{h_{t,i}}=x_{\infty}$, one has  
\begin{align*}
\lim_{t\to\infty, t\in N_2}T_{w_{h_{t,i}}}(x^{h_{t,i}})&=\lim_{t\to\infty, t\in N_2}\sum_{j\in I}w_{h_{t,i}}(j)T_j(x^{h_{t,i}})\\
&=\sum_{j\in I}w_{\infty,i}(j)T_j(x_{\infty})=T_{w_{\infty,i}}(x_{\infty})
\end{align*}
because of the continuity of the operators $T_j$, $j\in I$ and the limits $\lim_{t\to\infty, t\in N_2}x^{h_{t,i}}=x_{\infty}$ and $w_{\infty,i}(j)=\lim_{t\to\infty, t\in N_2}w_{h_{t,i}}(j)$ for all $j\in I$. These equalities, together with the equality $\lim_{t\to\infty, t\in N_2}\|T_{w_{h_{t,i}}}(x^{h_{t,i}})-x^{h_{t,i}}\|=0$, imply that $T_{w_{\infty,i}}(x_{\infty})=x_{\infty}$. But $w_{\infty,i}(i)\geq \alpha>0$ because $w_{h_{t,i}}(i)\geq \alpha$ for all $t\in \N\cup\{0\}$ by the choice of the sequence $\{h_{t,i}\}_{t=0}^{\infty}$. Hence, $i\in \wh{I}_{w_{\infty,i}}$ and, therefore, Corollary \bref{cor:T_w(x)=x} implies that $T_i(x_{\infty})=x_{\infty}$, namely $x_{\infty}\in F_i$. 

Thus, $x_{\infty}\in F_i$ if either $\{L(x^{h_{t,i}},w_{h_{t,i}})\}_{t\in N_1}$ is bounded or if it is unbounded. Since $i$ was an arbitrary index in $I$, we conclude from the previous paragraphs that $x_{\infty}\in \cap_{i\in I}F_i=F$, as claimed. 
\end{proof}

Now we are able to formulate and prove the convergence theorems.

\begin{thm}\label{thm:NonemotyInterior}
Suppose that $X$ is a finite-dimensional real Hilbert space and that $\left\{ x^{k}\right\}_{k=0}^{\infty}$ is a sequence generated by Algorithm \bref{alg:Extrapolation}, where $\{T_{i}\}_{i\in I}$ are finitely many continuous cutters defined on $X$ with fixed point sets $\{F_{i}\}_{i\in I}$ and a common fixed point set $F:=\cap_{i\in I}F_{i}$, where $\{w_{k}\}_{k=0}^{\infty}$ is a sequence of weight functions with respect to $I:=\{1,2,\ldots,m\}$ ($m\in\N$),  where $\tau_{1}$ and $\tau_{2}$ are in the interval $(0,1]$, and where $\{\lambda_{k}\}_{k=0}^{\infty}$ satisfy $\lambda_k\in \left[ \tau_{1},\left( 2-\tau _{2}\right) L\left( x^{k},w_{k}\right) \right]$ for all $k\in\N\cup\{0\}$ (and $L$ is defined in \beqref{eq:L(x,w)}). If the interior of $F$ is nonempty and Condition \bref{cond:sum_infty} holds, then $\left\{ x^{k}\right\}_{k=0}^{\infty}$ converges to a point $x^{\ast}\in F$.
\end{thm}

\begin{proof}
Proposition \bref{prop:FejerMonotone} ensures that $\left\{ x^{k}\right\}_{k=0}^{\infty}$ is Fej\'er monotone with respect to $F$. Since  the interior of $F$ is nonempty, Proposition \bref{prop:FejerProperties}\beqref{item:Interior==>Converges}  ensures that $\{x^{k}\}_{k=0}^{\infty}$ converges to a point $x^*\in X$, and since Condition \bref{cond:sum_infty} holds, Proposition \bref{P c 2.1.8/2} ensures that $x^*\in F$, as required. 
\end{proof}

\begin{thm}\label{thm:AlmostCyclic}
Suppose that $X$ is a finite-dimensional Hilbert space and that $\left\{ x^{k}\right\}_{k=0}^{\infty}$ is a sequence generated by Algorithm \bref{alg:Extrapolation}, where $\{T_{i}\}_{i\in I}$ are finitely many continuous cutters defined on $X$ with fixed point sets $\{F_{i}\}_{i\in I}$ and a common fixed point set $F:=\cap_{i\in I}F_{i}$, where $\{w_{k}\}_{k=0}^{\infty}$ is a sequence of weight functions with respect to $I:=\{1,2,\ldots,m\}$ ($m\in\N$),  where $\tau_{1}$ and $\tau_{2}$ are in the interval $(0,1]$, and where $\{\lambda_{k}\}_{k=0}^{\infty}$ satisfy $\lambda_k\in \left[ \tau_{1},\left( 2-\tau _{2}\right) L\left( x^{k},w_{k}\right) \right]$ for all $k\in\N\cup\{0\}$ (and $L$ is defined in \beqref{eq:L(x,w)}). If Condition \bref{cond:sw} holds, then $\left\{ x^{k}\right\}_{k=0}^{\infty}$ converges to a point $x^{\ast}\in F$.
\end{thm}
\begin{proof}
Proposition \bref{prop:FejerMonotone} ensures that the algorithmic sequence $\{x^k\}_{k=0}^{\infty}$ is Fej\'er monotone  with respect to $F$, and hence Parts \beqref{item:WeakClusterPoint} and \beqref{item:FiniteDimensional} in Proposition \bref{prop:FejerProperties} ensure that $\{x^k\}_{k=0}^{\infty}$ has at least one (strong) cluster point. Any such cluster point belongs to $F$ according to Proposition  \bref{prop:AccumulationFixedPoint_cond_sw}. Therefore, Parts \beqref{item:Cluster==>Converges} and  \beqref{item:FiniteDimensional} in Proposition \bref{prop:FejerProperties} imply that  $\{x^k\}_{k=0}^{\infty}$ converges (strongly) to some $x^{\ast}\in F$. 
\end{proof}

\section{Concluding remarks}\label{sec:Conclusion}
In this work we presented new results related to a wide class of operators called cutters, and used these results for analyzing an extrapolated block-iterative method aimed at solving the common fixed point problem (CFPP) induced by a finite collection of continuous cutters defined on a finite-dimensional Hilbert space. We showed that the algorithmic sequence produced by the method converges to a solution of the problem under conditions on the dynamic weights which have not been discussed so far in the context of extrapolated algorithms for solving the CFPP induced by cutters. An essential tool in the derivation of our results is the product space formalism of Pierra, a theory which we extended and also clarified a certain issue in it. Since the CFPP has numerous applications in science and engineering, and some instances of it have  received a lot of attention over the years, the ideas and results discussed here can be used in various settings and applications, possibly beyond the ones discussed in this work.

\section*{Acknowledgments}
D.R. thanks Christian Bargetz for a helpful discussion regarding \cite{BargetzKolobovReichZalas2019jour}. All authors thank the referees for their comments which helped to improve the text. This work is supported by U.S. National Institutes of Health (NIH) Grant Number R01CA266467 and by the Cooperation Program in Cancer Research of the German Cancer Research Center (DKFZ) and Israel’s Ministry of Innovation, Science and Technology (MOST).

%\bibliographystyle{acm}
%\bibliography{biblio}

%%%%%%%%%%%%%%%%%%%%%%%%%%%%%%%%%%%%%%%%%%%%%%%%
%\begin{comment}

%\end{comment}
%%%%%%%%%%%%%%%%%%%%%%%%%%%%%%%%%%%%%%%%%%%%%%%%

\end{document}